%% file: cocycle-Mar-5-13.tex
\documentclass[11pt,letterpaper]{article}
\usepackage{amsfonts, amsmath, amssymb, amscd, amsthm, graphicx, mathrsfs, wasysym, setspace, mdwlist, color}
\usepackage{hyperref}
\makeatletter
\def\blfootnote{\xdef\@thefnmark{}\@footnotetext}
\makeatother

\newtheorem{thm}{Theorem}[section]
\newtheorem{cor}[thm]{Corollary}
\newtheorem{lem}[thm]{Lemma}
\newtheorem{prop}[thm]{Proposition}

\theoremstyle{definition}
\newtheorem{defn}[thm]{Definition}
\newtheorem{ex}[thm]{Example}
\theoremstyle{remark}
\newtheorem{rem}[thm]{Remark}

\newcommand{\G }{\Gamma (G, X\sqcup \mathcal H)}
\newcommand{\dxh }{{\rm d}_{X\cup\mathcal H}}

\newcommand{\dl }{\widehat{\rm d}_{\lambda}}
\newcommand{\Ker }{{\rm Ker\, }}
\newcommand{\e }{\varepsilon }
\renewcommand{\kappa }{\varkappa}

\newcommand{\Hl }{\{ H_\lambda \} _{\lambda \in \Lambda } }

\renewcommand{\d }{{\rm d} }
\newcommand{\im}{{\rm Im\, }}

\newcommand{\he }{hyperbolically embedded }
\newcommand{\Lab }{{\bf Lab}}
\newcommand{\h}{\hookrightarrow _{h}}

\newcommand{\QZ}{{QZ^1}}
\newcommand{\QZas}{{QZ^1_{as}}}

\begin{document}

\title{Induced quasi-cocycles on groups with hyperbolically embedded subgroups}
\author{M. Hull, D. Osin\thanks{The research was supported by the NSF grant DMS-1006345. The second author was also supported by the RFBR grant 11-01-00945.
}}
\date{}
\maketitle

\begin{abstract}
Let $G$ be a group, $H$ a hyperbolically embedded subgroup of $G$, $V$ a normed $G$-module, $U$ an $H$-invariant submodule of $V$. We propose a general construction which allows to extend $1$-quasi-cocycles on $H$ with values in $U$ to $1$-quasi-cocycles on $G$ with values in $V$. As an application, we show that every group $G$ with a non-degenerate hyperbolically embedded subgroup has $\dim H_b^2(G,\ell^p(G))=\infty $ for $p\in [1, +\infty)$. This covers many previously known results in a uniform way. Applying our extension to quasimorphisms and using Bavard duality, we also show that hyperbolically embedded subgroups are undistorted with respect to the stable commutator length.
\end{abstract}


\section{Introduction}

Let $\mathbb F$ be a subfield of $\mathbb C$. All modules in this paper are left, all vector spaces are over $\mathbb F$. For a discrete group $G$, by a \emph{normed $G$-module} we mean a normed vector space $V$ endowed with a (left) action of the group $G$ by isometries. Given a subgroup $H\le G$, by an $H$-submodule of a $G$-module $V$ we mean any $H$-invariant subspace of $V$ with the induced action of $H$.

Let $V$ be a normed $G$-module. Recall that a map $q\colon G\to V$ is called a \emph{1-quasi-cocycle} if there exists a constant $\e>0$ such that for every $f,g\in G$ we have $$\| q(fg)-q(f)-fq(g)\|\le \e .$$  The $\mathbb F$-vector space of all $1$-quasi-cocycles on $G$ with values in $V$ is denoted by $\QZ (G,V)$.

The study of $1$-quasi-cocycles is partially motivated by the fact that the kernel of the comparison map  $H_b^2(G,V)\to H^2(G,V)$ from the second bounded cohomology to the ordinary second cohomology with coefficients in $V$ can be identified with the quotient  $\QZ(G,V)/(\ell^{\infty}(G,V)+Z^1(G,V))$, where $\ell^{\infty}(G,V)$ and $Z^1(G,V)$ are the subspaces of uniformly bounded maps and cocycles, respectively. In the last decade, techniques based on 1-quasi-cocycles and bounded cohomology have led to new breakthroughs in the study of rigidity of group von Neumann algebras, measure equivalence and orbit equivalence of groups, and low dimensional topology (see \cite{DC,CS,Mon07,Pop} and references therein).

The main goal of this paper is to address the following ``extension problem": \emph{Under what conditions can a 1-quasi-cocycle on a subgroup $H\le G$ be extended to the whole group $G$?}

Below we describe few known results in this direction.

\begin{ex}[Counting quasimorphisms]\label{ex1}
If $V=\mathbb R$ with the trivial action of $G$, 1-quasi-cocycles on $G$ with values in $V$ are called \emph{quasimorphisms}. The classical examples are \emph{counting quasimorphisms} of free groups introduced by Brooks \cite{Bro}. Let $F$ be a free group with a basis $S$ and let $w$ be a reduced word in $S\cup S^{-1}$. Given an element $f\in F$, denote by $c_w(g)$ the number of disjoint copies of $w$ in the reduced representative of $g$. Then $h_w=c_w-c_{w^{-1}}$ defines a quasimorphism $F\to \mathbb R$ \cite{Bro}. Observe that $h_w(g)$ extends the obvious cocycle (i.e., homomorphism) $H\to \mathbb R$ of the cyclic subgroup $H=\langle w\rangle \le F$ that sends $w^n$ to $n$ for all $n\in \mathbb Z$.

This construction was further developed by Epstein and Fujiwara \cite{EF} and later by Bestwina and Fujiwara \cite{BF}, who generalized it to the cases of hyperbolic groups and groups acting weakly properly discontinuous on hyperbolic spaces, respectively.
\end{ex}

Recall that a 1-quasi-cocycle $q\in \QZ(G,V)$ is called \emph{anti-symmetric} if
$$
q(g^{-1})=-g^{-1}q(g)
$$
for every $g\in G$. The next example is essentially due to Thom (cf. \cite[Lemma 5.1]{Thom}).

\begin{ex}[Extending anti-symmetric 1-quasi-cocycles to free products]\label{ex2}
Let $G=H_1\ast H_2$, let $V$ be a normed $G$-module, and let $U_i$ be an $H_i$-submodule of $G$, $i=1,2$. Then any anti-symmetric 1-quasi-cocycles $q_i\in \QZ (H_i, U_i)$, $i=1,2$, can be naturally extended to a 1-quasi-cocycle $G\to V$ using the normal form of elements of free products. That is, suppose that $g\in G$ has the normal form
$$
g=h_1k_1\cdots h_nk_n,
$$
where $h_i\in H_1$, $k_i\in H_2$ for $i=1, \ldots, n$, and $k_1,h_2, \ldots , k_{n-1},h_n$ are non-trivial. Let
$$
q(g)= q_1(h_1) + h_1 q_2(k_1) + h_1k_1q_1(h_2)+\cdots + h_1k_1\cdots  h_n q_2(k_n).
$$
Checking that $q\in \QZ(G,V)$ is easy. It is essential here that $q_1$ and $q_2$ are anti-symmetric (see Remark \ref{as-nec}).
\end{ex}

\begin{ex}[No general extension construction exists]
It is well-known and easy to prove that every quasimorphism on an amenable group decomposes as a sum of a homomorphism and a bounded map \cite{DC}. This easily implies that if $G$ is amenable and $H=\langle h\rangle \le [G,G]$ is an infinite cyclic subgroup, then the natural homomorphism $H\to \mathbb R$ defined by $h^n\mapsto n$ does not extend to any quasimorphism of $G$.
\end{ex}

In this paper we prove an extension theorem which can be thought of as a generalization of Examples \ref{ex1} and \ref{ex2}. In fact, our construction is similar to Example \ref{ex2}, but the proof is much more involved. We state here a simplified version of our main result and refer to Theorem \ref{real} for the full generality. For a group $G$ and a normed $G$-module $V$, let $\QZas(G,V)$ denote the subspace of all anti-symmetric 1-quasi-cocycles on $G$ with coefficients in $V$.

\begin{thm}\label{ind}
Let $G$ be a group, $H$ a hyperbolically embedded subgroup of $G$, $V$ a normed $G$-module, $U$ an $H$-submodule of $V$.  Then there exists a linear map $$\iota \colon \QZas (H, U)\to \QZas (G,V)$$ such that for any $q\in \QZas(H,U)$, we have $\iota (q)|_H\equiv q.$
\end{thm}

It is well-known and easy to prove that every 1-quasi-cocycle is anti-symmetric up to a bounded perturbation (see Lemma \ref{antisym}). In the notation of Theorem \ref{ind}, this gives the following.

\begin{cor}\label{ind1}
There exists a linear map $\kappa\colon \QZ (H, U)\to \QZ (G,V)$ such that for any $q\in \QZ(H,U)$, $\kappa (q)|_H\in \QZ(H,U)$  and $$\sup_{h\in H} \| \kappa (q) (h)-q(h)\| <\infty .$$
\end{cor}

The notion of a hyperbolically embedded subgroup of a group was introduced in \cite{DGO} and encompasses many examples of algebraic and geometric nature. We discuss some of them here and refer to the next section and \cite{DGO} for the definition and details.

\begin{enumerate}
\item[(a)] Let $G$ be any group and let $H\le G$ be a finite subgroup or $H=G$. Then $H$ is
\he in $G$. In what follows these cases are referred to as \emph{degenerate.}

\item[(b)] Let $ G$ be a group hyperbolic relative to a collection of peripheral subgroups $\Hl$. Then every peripheral subgroup is \he in $G$. In particular, if $G=H_1\ast H_2$, then $H_1$ and $H_2$ are \he in $G$.

\item[(c)] Let $G$ be a relatively hyperbolic group and let $g$ be a loxodromic element. Then $g$ is contained in the unique maximal virtually cyclic subgroup $E(g)$ of $G$ and $E(g)$ is \he in $G$ \cite{Osi06}. In particular, this holds for every infinite order element $g$ of a hyperbolic group $G$.
\item[(d)] More generally, let $G$ be a group acting on a hyperbolic space and containing a loxodromic element $g$ that satisfies the Bestvina-Fujiwara WPD condition (see \cite{BF} or \cite{DGO} for the definition). Then $g$ is contained in the unique maximal virtually cyclic subgroup $E(g)$ of $G$ and $E(g)$ is \he in $G$ \cite[Theorem 6.8]{DGO}. This general result applies in the following cases: (d$_1$) $G$ is the mapping class group of a punctured closed orientable surface and $g$ is a pseudo-Anosov element \cite{BF}; (d$_2$) $G=Out(F_n)$ and $g$ is a fully irreducible automorphism \cite{BFe}.
\item[(e)] Similarly to the previous example, let $G$ be a group acting properly on a proper $CAT(0)$ space and let $g$ be a rank-1 element. Then $g$ is contained in the unique maximal virtually cyclic subgroup $E(g)$ of $G$ and $E(g)$ is \he in $G$ \cite{Sis}.
\end{enumerate}

\begin{ex}[cf. \cite{BF,CF}]
Let us illustrate our theorem by extending quasimorphisms in the case when $G$ and $g$ are as in examples (c), (d), or (e) above. It is well known and easy to prove that every infinite virtually cyclic group is either finite-by-(infinite cyclic) or finite-by-(infinite dihedral). If $E(g)$ is of the former type, there exists a homomorphism $q\colon E(g)\to \mathbb R$ that extends the natural map $g^n \to n$. By our theorem, $q$ extends to a quasimorphism of $G$, which can be thought of as a generalization of the Brooks' counting quasimorphism. In particular, such quasimorphisms can always be constructed if $G$ has no involutions.

On the other hand, if $E(g)$ is finite-by-(infinite dihedral), then it is easy to show that there exists $a\in G$ and  $n\in \mathbb N$ such that $$a^{-1}g^na=g^{-n}.$$ This equality implies that every quasimorphism $E(g)\to \mathbb R$ is bounded. Thus no analogue of the counting quasimorphism exists in this case.
\end{ex}

In Section \ref{sepco}, we develop the main idea in the construction of our extension, which is the notion of seperating cosets of a subgroup $H$ which is hyperbolically embedded in $G$.  This allows use to associate a canonical, finite set of $H$-cosets to each $g\in G$, and to each such coset a finite collection of $h\in H$. This is essentially what is given by the normal forms of elements in Example \ref{ex2}, and we are then able to extend quasi-cocycles in a similiar manner. The main technical tool in proving that our extension actually gives a quasi-cocycle is the decompostion of the separating cosets of a triangle in Lemma \ref{fgh}.

In Section \ref{app}, we obtain some other corollaries of our main result. Recall that the class $\mathcal C_{reg}$ of Monod-Shalom is the class of groups for which $H_b^2(G, \ell^2(G))\ne 0$. This definition was proposed as cohomological characterization of the notion of ``negative curvature" in group theory \cite{MS2}. In \cite{MS} Monod and Shalom develop a rich rigidity theory with respect to measure equivalence and orbit equivalence of actions of groups in $\mathcal C_{reg}$. These results have a variety of applications to measurable group theory, ergodic theory and von Neumann algebras.

Another similar class of groups is the class $\mathcal D_{reg}$ introduced by Thom \cite{Thom}. $G\in \mathcal D_{reg}$ if $G$ is non-amenable and there exists some $q\in \QZ(G,\ell^2(G))$ which is unbounded. Thom proved rigidity results about the subgroup structure of groups in $\mathcal D_{reg}$ and showed that this class is closely related to $\mathcal C_{reg}$. However neither inclusion is known to hold between these two classes.

Let $\mathcal X$ denote the class of groups with non-degenerate hyperbolically embedded subgroups. Using Corollary \ref{ind1} and the fact that every group $G\in \mathcal X$ contains a virtually free (but not virtually cyclic) \he subgroup \cite{DGO}, we prove the following.

\begin{cor}\label{cor1}
For any $G\in \mathcal X$, the dimension of the kernel of the comparison map $H_b^2(G, \ell^p(G))\to H^2(G, \ell^p(G))$ is infinite. In particular, $\mathcal X \subseteq \mathcal C_{reg}\cap \mathcal D_{reg}$.
\end{cor}

This corollary recovers several previously known results in a uniform way. For example, this was known for hyperbolic groups \cite{MMS} and more generally groups acting non-elementary and acylindrically on hyperbolic spaces \cite{Ham1}, groups acting properly on proper $CAT(0)$ spaces and containing a rank-$1$ isometry \cite{Ham2}, and $Out(F_n)$ for $n\ge 2$ \cite{Ham3}. All of these groups belong to $\mathcal X$ \cite{DGO}.

At the final stage of our work we learned that Bestvina, Bromberg, and Fujiwara  \cite{BBF} independently and simultaneously proved that the dimension of the kernel of the comparison map $H_b^2(G, E)\to H^2(G, E)$ is infinite for any group acting non-elementary on a hyperbolic space and containing a WPD loxodromic isometry and any uniformly convex Banach $G$-module $E$. In fact, the class of groups acting non-elementary on a hyperbolic space and containing a WPD loxodromic isometry coincides with our class $\mathcal X$ (see Theorem 6.8 and Corollary 6.10 in \cite{DGO}). Thus the result of Bestvina, Bromberg, and Fujiwara is stronger than Corollary \ref{cor1}.

As another application, we show that hyperbolically embedded subgroups are undistorted with respect to the stable commutator length, $scl$. For the definition of $scl$ we refer to Section \ref{app}. Given a group $G$ and a subgroup $H\le G$ it is straightforward to see that $scl _G(h)\le scl_H(h)$ for any $h\in [H,H]$, where $scl_G$ and $scl_H$ are the stable commutator lengths on $[G,G]$ and $[H,H]$, respectively.

On the other hand, recall that every torsion free group $H$ can be embedded in a group $G$ where every element is a commutator (see \cite[Theorem 8.1]{LS} or \cite{Osi10} for a finitely generated version of such an embedding). In particular, $scl_G$ vanishes on $G$, while $scl_H$ can be unbounded on $[H,H]$. Thus, in general, there  is no upper bound on $scl_H$ in terms of $scl_G$. In what follows, we say that $H$ is \emph{undistorted in $G$ with respect to the stable commutator length} if there exists a constant $B$ such that for every $h\in [H,H]$, we have $scl_H(h)\le Bscl _G(h).$

Using Theorem \ref{ind} and the Bavard duality, we obtain the following.

\begin{cor}\label{scl}
Let $G$ be a group, $H$ a \he subgroup of $G$. Then $H$ is undistorted in $G$ with respect to the stable commutator length.
\end{cor}

Even the following particular cases seem new. Recall that a subgroup $H\le G$ is \emph{almost malnormal} if $|H^g\cap H|<\infty $ for every $g\in G\setminus H$.

\begin{cor}\label{malnorm}
Every almost malnormal quasiconvex subgroup of a hyperbolic group is undistorted with respect to the stable commutator length. In particular, so is every finitely generated malnormal subgroup of a free group.
\end{cor}

In Section \ref{app} we show that the almost malnormality condition can not be omitted even for free groups (see Remark \ref{free-dist}).

\section{Preliminaries}

\paragraph{Notation and conventions.} In this paper we allow length functions and metrics to take infinite values. In particular, the word length $|\cdot |_S$ on a group $G$ corresponding to a (not necessary generating) set $S$ is defined by letting $|g|_S$ be the length of a shortest word in $S\cup S^{-1}$ representing $g$ if $g\in \langle S\rangle $ and $|g|_S=\infty $ otherwise. The corresponding metric on $G$ is denoted by $\d_S$; thus $\d_S(f,g)=|f^{-1}g|_S$.

By a path $p$ in a (Cayley) graph we always mean a combinatorial path; we denote the label of $p$ by $\Lab (p)$ and we denote the origin and terminus of $p$ by $p_-$ and $p_+$ respectively.

For the rest of the paper, we will refer to $1$-quasi-cocycles simply as quasi-cocycles.

\paragraph{Hyperbolically embedded subgroups}
Let $G$ be a group, $\Hl $ a collection of subgroups of $G$.
Let
\begin{equation}\label{calH}
\mathcal H= \bigsqcup\limits_{\lambda \in \Lambda } (H_\lambda\setminus \{ 1\}).
\end{equation}
Given a subset $X\subseteq G$ such that $G$ is generated by $X$ together with the union of all $H_\lambda$'s,  we denote by $\G $ the Cayley graph of $G$ whose edges are labeled by letters from the alphabet $X\sqcup\mathcal H$. That is, two vertices $g,h\in G$ are connected by an edge going from $g$ to $h$ and labeled by $a\in X\sqcup\mathcal H$ iff $a$ represents the element $g^{-1}h$ in $G$. Note that some letters from $X\sqcup\mathcal H$ may represent the same element in $G$, in which case $\G $ has multiple edges corresponding to these letters.

We think of the Cayley graphs $\Gamma (H_\lambda, H_\lambda \setminus\{1\})$ as (complete) subgraphs of $\G $. For every $\lambda \in \Lambda $, we introduce a \textit{relative metric} $\dl \colon H_\lambda \times H_\lambda \to [0, +\infty]$ as follows. Given $h,k\in H_\lambda $, let $\dl (h,k)$ be the length of a shortest path in $\G $ that connects $h$ to $k$ and has no edges in $\Gamma (H_\lambda, H_\lambda \setminus\{1\})$. If no such a path exists, we set $\dl (h,k)=\infty $. Clearly $\dl $ satisfies the triangle inequality. In case the collection consists of a single subgroup $H\le G$, we denote the corresponding relative metric on $H$ simply by $\widehat d$.

\begin{defn}
Let $G$ be a group, $X$ a (not necessary finite) subset of $G$. We say that a collection of subgroups $\Hl$ of $G$ is \emph{\he in $G$ with respect to $X$} (we write $\Hl \h (G,X)$) if the following conditions hold.
\begin{enumerate}
\item[(a)] The group $G$ is generated by $X$ together with the union of all $H_\lambda$ and the Cayley graph $\G $ is hyperbolic.
\item[(b)] For every $\lambda\in \Lambda $, $(H_\lambda, \dl )$ is a locally finite metric space; that is, any ball of finite radius in $H_\lambda $ contains finitely many elements.
\end{enumerate}
Further we say  that $\Hl$ is hyperbolically embedded in $G$ and write $\Hl\h G$ if $\Hl\h (G,X)$ for some $X\subseteq G$.
\end{defn}
\begin{figure}
  \centering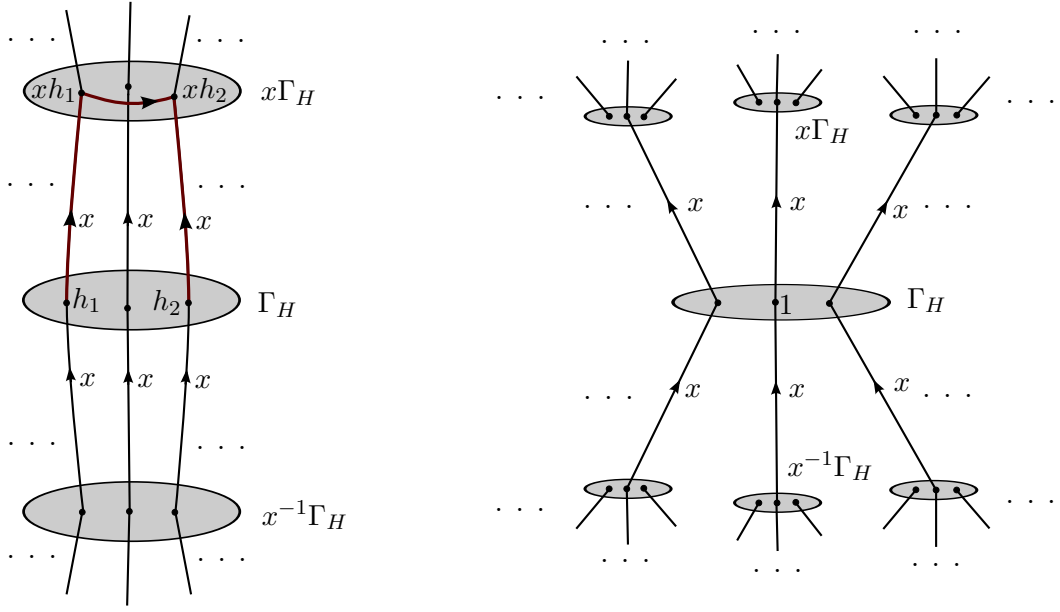\\
  \caption{Cayley graphs $\Gamma(G, X\sqcup H)$ for $G=H\times \mathbb Z$ and $G=H\ast \mathbb Z$.}\label{fig0}
\end{figure}

\begin{ex}
\begin{enumerate}
\item[(a)] Let $G$ be any group. Then $G\h G$.  Indeed take $X=\emptyset $. Then the Cayley graph $\Gamma(G, X\sqcup H)$ has diameter $1$ and $d(h_1, h_2)=\infty $ whenever $h_1\ne h_2$. Further, if $H$ is a finite subgroup of a group $G$, then $H\h G$. Indeed $H\h (G,X)$ for $X=G$. These cases are referred to as {\it degenerate}. In what follows we are only interested in non-degenerate examples.

\item[(b)] Let $G=H\times \mathbb Z$, $X=\{ x\} $, where $x$ is a generator of $\mathbb Z$. Then $\Gamma (G, X\sqcup H)$ is quasi-isometric to a line and hence it is hyperbolic. However $\widehat\d(h_1, h_2)\le 3$ for every $h_1, h_2\in H$. Indeed let $\Gamma _H$ denote the Cayley graph $\Gamma (H, H\setminus \{ 1\}) $. In the shift $x\Gamma _H$ of $\Gamma _H$ there is an edge (labeled by $h_1^{-1}h_2\in H$) connecting $h_1x$ to $h_2x$, so there is a path of length $3$ connecting $h_1$ to $h_2$ and having no edges in $\Gamma _H$ (see Fig. \ref{fig0}). Thus if $H$ is infinite, then $H\not\h (G,X)$. Moreover, a similar argument shows that $H\not\h G$.

\item[(c)] Let $G=H\ast \mathbb Z$, $X=\{ x\} $, where $x$ is a generator of $\mathbb Z$. In this case $\Gamma (G, X\sqcup H)$ is quasi-isometric to a tree (see Fig. \ref{fig0}) and $\widehat\d(h_1, h_2)=\infty $ unless $h_1=h_2$. Thus $H\h (G,X)$.
\end{enumerate}
\end{ex}

It turns out that the relative metric $\dl$ can be realized as a word metric with respect to some finite set.

\begin{lem}[{\cite[Lemma 4.10]{DGO}}]\label{relmet}
\label{Y}
Let $\Hl\h G$. Then for each $\lambda\in\Lambda$, there exists a finite subset $Y_{\lambda}\subseteq H_{\lambda}$ such that $\dl$ is bi-Lipschitz equivalent to the word metric with respect to $Y_{\lambda}$. That is, for $h_1,h_2\in H_\lambda $, $\dl (h_1,h_2)$ is finite if and only if $\d_{Y_\lambda}(h_1,h_2)$ is, and the ratio $\dl/\d_{Y_\lambda}$ is uniformly bounded on $H_\lambda\times H_\lambda$ minus the diagonal.
\end{lem}

\paragraph{Components.} Let $\Hl \h (G,X)$. Let $q$ be a path in the Cayley graph $\G $. A (non-trivial) subpath $p$ of $q$ is called an \emph{$H_\lambda $-subpath}, if the label of $p$ is a word in the alphabet $H_\lambda \setminus\{ 1\} $. An $H_\lambda$-subpath $p$ of $q$ is an {\it $H_\lambda $-component} if $p$ is not contained in a longer $H_\lambda$-subpath of $q$; if $q$ is a loop, we require in addition that $p$ is not contained in any longer $H_\lambda $-subpath of a cyclic shift of $q$. Further by a {\it component} of $q$ we mean an $H_\lambda $-component of $q$ for some $\lambda \in \Lambda$.

Two $H_\lambda $-components $p_1, p_2$ of a path $q$ in $\G $ are called {\it connected} if there exists a
path $c$ in $\G $ that connects some vertex of $p_1$ to some vertex of $p_2$, and ${\Lab (c)}$ is a word consisting only of letters from $H_\lambda \setminus\{ 1\} $. In algebraic terms this means that all vertices of $p_1$ and $p_2$ belong to the same left coset of $H_\lambda $. Note also that we can always assume that $c$ has length at most $1$ as every non-trivial element of $H_\lambda $ is included in the set of generators.

It is convenient to extend the metric $\dl $ defined above to the whole group $G$ by assuming $\dl (f,g)\colon =\dl (f^{-1}g,1)$ if $f^{-1}g\in H_\lambda $ and $\dl (f,g)=\infty $ otherwise. One important technical tool is the following corollary of (a particular case of) \cite[Proposition 4.13]{DGO}.

\begin{lem}\label{C}
There exists a constant $C>0$ such that for any geodesic $n$-gon $p$ in $\G$  and any isolated component $a$ of $p$, we have $\dl (a_-, a_+)\le Cn$.
\end{lem}

\begin{proof}
Let $p=p_1\ldots p_n$, where $p_1, \ldots, p_n$ are geodesic. For definiteness, suppose that $a$ is a component of $p_1$, i.e., $p_1=qar$. By \cite[Proposition 4.13]{DGO} applied to the $(n+2)$-gon $qarp_2\ldots p_n$, we have $\dl (a_-, a_+)\le D(n+2)\le 2D n$, where $D$ is a constant independent of $n$ ($D=D(1,0)$ in the notation of \cite[Proposition 4.13]{DGO}). It remains to take any positive $C\ge 2D$.
\end{proof}

\paragraph{Quasi-cocycles.}
For a quasi-cocycle $q\in \QZ(G,V)$ we define its \emph{defect} $D(q)$ by
\begin{equation}\label{Dq}
D(q)=\sup\limits_{f,g\in G} \| q(fg)-q(f)-fq(g)\|.
\end{equation}
Note that
\begin{equation}\label{q1}
\| q(1)\| = \| q(1\cdot 1)-q(1)-1q(1)\| \le D(q).
\end{equation}

We will use the following elementary fact.

\begin{lem}\label{antisym}
Let $G$ be a group, $V$ a $G$-module.
Then there exists a linear map $$\alpha \colon \QZ(G,V)\to QZ_{as}^1(G,V)$$ such that for every $q\in \QZ(G,V)$ we have $$ \sup_{g\in G} \| \alpha (q)(g)-q(g)\|< D(q).$$
\end{lem}
\begin{proof}
Take $\alpha (q)(g)=\frac12 (q(g) - gq(g^{-1})) $. Verifying all properties is straightforward. Indeed for every $g\in G$, we have
$$
\| \alpha (q)(g)-q(g)\| = \frac{1}{2}\| -q(g)-gq(g^{-1})\| \le \frac{1}{2}\| q(1)-q(g)-gq(g^{-1})\| +\frac{1}{2}\| q(1)\| \le D(q),
$$
where the last inequality uses (\ref{q1}).  Further,
$$
\alpha (q)(g^{-1}) = \frac12 (q(g^{-1})-g^{-1}q(g))=\frac12 g^{-1} (gq(g^{-1}-q(g))=-g^{-1} \alpha (q)(g).
$$
\end{proof}

\paragraph{Bounded cohomology.} Recall the definition of the bounded cohomology of a (discrete) group $G$ with coefficients in an arbitrary normed $G$-module $V$. Let $C^n(G, V)$ be the vector space of $n$-cochains with coefficients in $V$, i.e., functions $G^n\to V$. The coboundary maps $d^n :C^n(G, V)\to C^{n+1}(G, V)$ are defined by the formula

\begin{align*}
d^n f(g_1, ..., g_{n+1})=& g_1f(g_2,...,g_{n+1})+\sum_{i=1}^{n}(-1)^if(g_1,...,g_{i-1},g_ig_{i+1},g_{i+2},...,g_{n+1})\\& +(-1)^{n+1} f(g_1,...,g_n).
\end{align*}

Let $Z^n(G, V)$ and $B^n(G, V)$ denote the cocycles and coboundaries of this complex respectively; that is, $Z^n(G, V)=\Ker d^n$ and $B^n(G, V)=\im d^{n-1}$ for $n>0$ and $B^0(G,V)=0$. Recall that the ordinary cohomology groups are defined by $$H^n(G,V)\colon =Z^n(G,V)/B^n(G,V).$$

Restricting to the subspaces $C_b^n(G, V)$ of $C^n(G, V)$ consisting of functions whose image is bounded with respect to the norm on $V$, we get the complex of bounded cochains. Similarly let $Z_b^n(G, V)$ and $B_b^n(G, V)$ denote its cocycles and coboundaries. Then the group $$H^n_b(G, V)\colon =Z_b^n(G, V)/B_b^n(G, V)$$ is called the \emph{$n$-th bounded cohomology group of $G$ with coefficients in $V$}.

Note that there is a natural map $c\colon H^n_b(G,V)\to H^n(G,V)$ which is induced by the inclusion map of the cochain complexes. This map is called the \emph{comparison map}, and the kernel of $c$ is denoted $EH_b^n(G, V)$. The following lemma is proved in \cite{Mon} (see also \cite{Thom}) in the case when $V$ is a Banach space. The same  proof works in the general case. We briefly sketch the argument for convenience of the reader.

\begin{lem}\label{seq}
Let $G$ be a discrete countable group, $V$ a normed $G$-module. Then there exists an exact sequence
$$0 \to \ell^{\infty}(G,V)+ Z^1(G,V) \to \QZ(G,V) \stackrel{\delta}{\to} H^2_b(G,V) \stackrel{c}{\to} H^2(G,V),$$
where $\ell^{\infty}(G,V)$ is the vector space of all uniformly bounded functions $G\to V$.
\end{lem}
\begin{proof}
We can identify $\QZ(G, V)$ with the subspace of $1$-cochains $q$ for which $d^1 q$ is uniformly bounded, that is $d^1q\in C_b^{2}(G, V)$. Since $d^2\circ d^1\equiv0$, $d^1 q$ is in fact a bounded 2-cocycle. Let $\delta\colon \QZ(G, V)\to H^2_b(G, V)$ denote the composition of $d^1$ and the natural quotient map $Z_b^{2}(G, V)\to H_b^{2}(G, V)$. Then $\delta q$ represents a trivial element of $H^2_b(G,V)$ if and only if $d^1 q=d^1 p$ for some bounded cochain $p$, which means $p\in \ell^{\infty}(G,V)$ and $q-p\in Z^1(G,V)$. Further if $q$ is a bounded 2-cocycle and $[q]_b\colon = q+B^2_b(G,V)\in H_b^2(G, V)$ is in the kernel of $c$, then $q=d^1 f$ for some 1-cochain $f$, which means $f\in\QZ(G, V)$ and $\delta f=[q]_b$.
\end{proof}

\section{Separating cosets}\label{sepco}

Throughout this section, we denote by $G$ a group with hyperbolically embedded collection of subgroups $\Hl\h G$. Let $X$ denote a subset of $G$ such that $\Hl\h (G,X)$. We also keep the notation $\mathcal H$ and $\G $ introduced in the previous section. By $\dxh$ we denote the word metric on $G$ with respect to the subset $X\sqcup \mathcal H$. By a coset of a subgroup we always mean a left coset.

We begin by introducing the notion of a separating coset for a pair of elements $f,g\in G$, which plays a crucial role in our construction.

\begin{defn}\label{sepcosdef}
We say that a path $p$ in $\G $ penetrates a coset $xH_\lambda$ for some $\lambda\in \Lambda $ if $p$ decomposes as $p_1ap_2$, where $p_1, p_2$ are possibly trivial, $(p_1)_+\in xH_\lambda$, and $a$ is an $H_\lambda$-component of $p$. If, in addition, $\dl (a_-, a_+)> 3C$,  where $C$ is the constant from Lemma \ref{C}, we say that $p$  \emph{essentially penetrates} $xH_\lambda$. Note that if $p$ is geodesic, it penetrates every coset of $H_\lambda$ at most once; in this case the vertices $a_-$ and $a_+$ are called the \emph{entrance and the exit points of $p$ in $xH_\lambda$} and are denoted by $p_{in} (xH_\lambda)$ and $p_{out}(xH_\lambda)$, respectively.

Given two elements $f,g\in G$, we denote by $\mathcal G(f,g)$ the set of all geodesics in $\G$ going from $f$ to $g$. Further we say that a coset $xH_\lambda $ is \emph{$(f,g)$-separating} if there exists a geodesic $p\in \mathcal G (f,g)$ that essentially penetrates $xH_\lambda $. For technical reasons we will also say $xH_\lambda$ is $(f, g)$-separating whenever $f$ and $g$ are both elements of $xH_\lambda$ and $f\ne g$; in this case we say $xH_\lambda$ is \emph{trivially $(f,g)$-separating}.  The set of all $(f,g)$-separating cosets of $H_\lambda$ is denoted by $S_\lambda (f,g)$.
\end{defn}

The following lemma immediately follows from the definition and the facts that if $f,g,h\in G$ and  $p\in \mathcal G(f,g)$, then $p^{-1}\in \mathcal G(g,f)$ and $hp\in \mathcal G(hf,hg)$.

\begin{lem}\label{Sfg}
For any $f,g,h\in G$ and any $\lambda \in \Lambda $, the following holds.
\begin{enumerate}
\item[(a)] $S_\lambda (f,g)=S_\lambda(g,f)$.
\item[(b)] $S_\lambda (hf,hg)=\{ hxH_\lambda \mid xH_\lambda \in S_\lambda (f,g)\}.$
\end{enumerate}
\end{lem}

The terminology in Definition \ref{sepcosdef} is justified by the first claim of following.

\begin{lem}\label{sep}
For any $\lambda \in \Lambda $, any $f,g\in G$ such that $f^{-1}g\notin H_\lambda $, and any $(f,g)$-separating coset $xH_\lambda $, the following hold.
\begin{enumerate}
\item[(a)] Every path in $\G$ connecting $f$ to $g$ and composed of at most $2$ geodesics penetrates $xH_\lambda$.
\item[(b)] For any $p,q\in \mathcal G(f,g)$, we have
$$
\dl (p_{in} (xH_\lambda ),q_{in} (xH_\lambda ))\le 3C
$$
and
$$
\dl (p_{out} (xH_\lambda ),q_{out} (xH_\lambda ))\le 3C.
$$
\end{enumerate}
\end{lem}

\begin{figure}
 \centering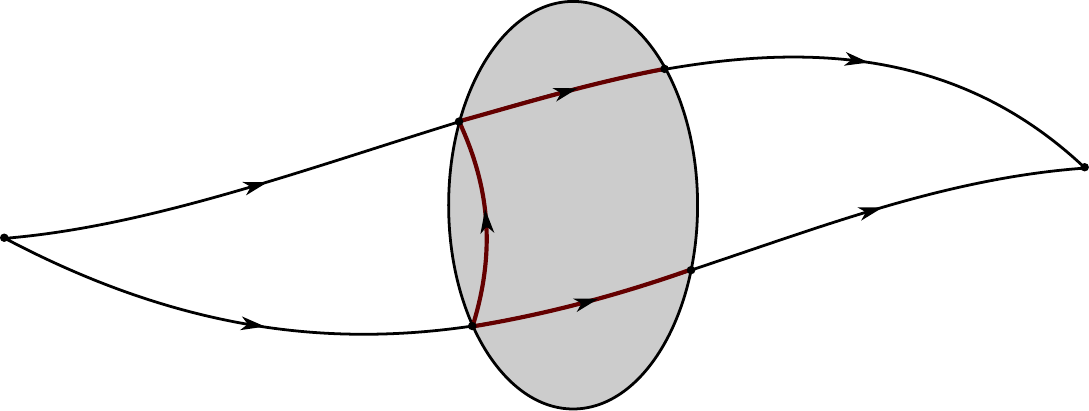\\
  \caption{}\label{fig3}
\end{figure}

\begin{proof}
Let $xH_\lambda \in S_\lambda (f,g)$ be $(f,g)$-separating coset. Since $f^{-1}g\notin H_\lambda$, $xH_\lambda $ is non-trivially separating. Thus there exists a geodesic $p\in \mathcal G(f,g)$ that essentially penetrates $xH_\lambda $; let $a$ denote the corresponding $H_\lambda $-component of $p$. Let $r$ be any other path in $\G$ connecting $f$ to $g$ and composed of at most $2$ geodesics. If $a$ is isolated in the loop $pr^{-1}$, we obtain $\dl (a_-, a_+)\le 3C$ by Lemma \ref{C}. This contradicts the assumption that $p$ essentially penetrates $xH_\lambda $. Hence $a$ is not isolated in $pr^{-1}$. Since $p$ is geodesic, $a$ cannot be connected to a component of $p$. Therefore $a$ is connected to a component of $r$, i.e. $r$ penetrates $xH_\lambda $.

Further let $p,q\in \mathcal G(a,b)$ and $xH_\lambda \in S_\lambda (f,g)$. By part (a) we have $p=p_1ap_2$ and $q=q_1bq_2$, where $(p_1)_+\in xH_\lambda$, $(q_1)_+\in xH_\lambda $ and $a$, $b$ are $H_\lambda$-components of $p$ and $q$, respectively (see Figure \ref{fig3}). (Of course, $p_i$ or $q_i$, $i=1,2$, can be trivial). Then $a$ and $b$ are connected. Let $e$ be an edge or the trivial path connecting $a_-$ to $b_-$ and labeled by a letter from $H_\lambda \setminus \{1\}$. Applying Lemma \ref{C} to the geodesic triangle $p_1eq_1^{-1}$, we obtain $\dl (e_-, e_+) \le 3C $, which gives us the first inequality in (b). The proof of the second inequality is symmetric.
\end{proof}

\begin{cor}\label{Sfin}
For any $f,g\in G$ and any $\lambda \in \Lambda $, we have $|S_\lambda (f,g)|\le \dxh (f,g)$. In particular, $S_\lambda (f,g)$ is finite.
\end{cor}

In this section we will use the following elementary observation several times. 

\begin{lem}\label{dist}
Let $p$ be a geodesic in $\G$. Suppose that $p$ penetrates a coset $xH_\lambda$. Let $p_0$ be the initial subpath of $p$ ending at $p_{in}(xH_\lambda)$. Then $\ell (p_0)=\dxh (p_-, xH_\lambda)$.
\end{lem}

\begin{proof}
Clearly $\dxh (p_-, xH_\lambda)\le \ell (p_0)$. Suppose that $\dxh (p_-, xH_\lambda)< \ell (p_0)$. Since  $xH_\lambda$ has diameter $1$ with respect to the metric $\dxh$, we obtain
$$
\dxh (p_-, p_{out}(xH_\lambda))\le \dxh (p_-, xH_\lambda)+1<\ell(p_0)+1.
$$
However we obviously have $\ell(p_0)+1=\dxh (p_-, p_{out}(xH_\lambda))$. A contradiction.
\end{proof}

\begin{defn}
Given any $f,g\in G$, we define a relation $\preceq$ on the set $S_\lambda (f,g)$ as follows:
$$xH_\lambda \preceq yH_\lambda\;\;\; {\rm  iff}\;\;\; \dxh (f, xH_\lambda )\le \dxh (f, yH_\lambda ).$$
\end{defn}

The next lemma is an immediate consequence of Lemma \ref{sep} and Lemma \ref{dist}.

\begin{lem}\label{sep-ref}
For any $f,g\in G$ and any $\lambda \in \Lambda $ with $f^{-1}g\notin H_\lambda $, $\preceq$ is a linear order on $S_\lambda (f,g)$ and every geodesic $p\in \mathcal G(f,g)$ penetrates all $(f,g)$-separating cosets according to the order $\preceq$. That is, $S_\lambda (f,g)=\{ x_1H_\lambda\preceq x_2H_\lambda \preceq \ldots \preceq x_nH_\lambda \}$ for some $n\in \mathbb N$ and $p$ decomposes as $$p=p_1a_1\cdots p_na_np_{n+1},$$ where $a_i$ is an $H_\lambda $-component of $p$ and $(p_i)_+\in x_iH_\lambda $ for $i=1, \ldots , n$ (see Fig. \ref{fig1}).
\end{lem}

\begin{figure}
 \centering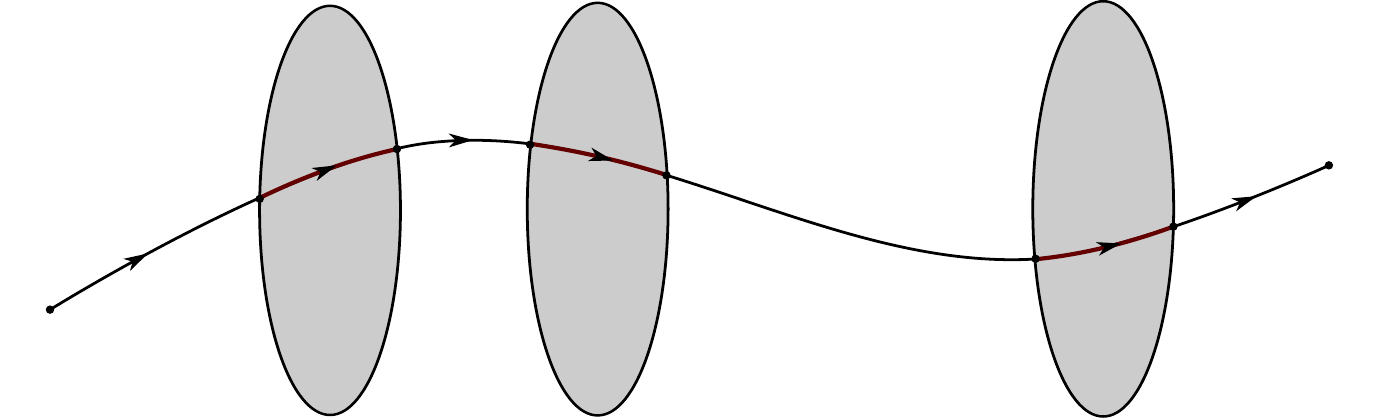\\
  \caption{}\label{fig1}
\end{figure}

Given $f,g\in G$ and $xH_\lambda \in S_\lambda (f,g)$, we denote by $E(f, g; xH_{\lambda})$ the set of ordered pairs of entrance-exit points of geodesics from $\mathcal G(f,g)$ in the coset $xH_{\lambda}$. That is,
$$
E(f, g; xH_{\lambda})=\{ (p_{in}(xH_\lambda), p_{out}(xH_\lambda))\mid p\in \mathcal G(f,g)\}.
$$

\begin{lem}\label{Efg}
For any $\lambda \in \Lambda $ and any $f,g,h,x\in G$, the following hold.
\begin{enumerate}
\item[(a)] $E(g,f; xH_\lambda )= \{ (v,u) \mid (u,v)\in E(f, g; xH_{\lambda})\} $.
\item[(b)] $E(hf,hg; xH_\lambda)= \{ (hu,hv) \mid (u,v)\in E(f, g; xH_{\lambda})\} $.
\item[(c)] $|E(f, g; xH_{\lambda})|<\infty $.
\end{enumerate}
\end{lem}
\begin{proof}
Parts (a) and (b) follow immediately from Lemma \ref{Sfg}. To prove (c), note that if $xH_\lambda$ trivially separates $f$ and $g$, then $E(f, g; xH_{\lambda})=\{ (f,g)\}$. Further if $xH_\lambda$ separates $f$ and $g$ non-trivially, fix any $(u,v)\in E(f, g; xH_{\lambda})$. Then for any other $(u^\prime, v^\prime)\in E(f, g; xH_{\lambda})$, we have $\dl (u,u^\prime)<3C$ and $\dl (v, v^\prime)<3C$ by part (b) of Lemma \ref{sep}. Recall that $(H_\lambda , \dl)$ is a locally finite metric space by the definition of a hyperbolically embedded collection of subgroups. Hence $|E(f, g; xH_{\lambda})|<\infty $.
\end{proof}

The main result of this section is the following.

\begin{lem}\label{fgh}
For any $f,g,h\in G$ and any $\lambda \in \Lambda $, the set of all $(f,g)$-separating cosets of $H_\lambda$ can be decomposed as $$S_\lambda (f,g)=S^{\prime}\sqcup S^{\prime\prime}\sqcup F,$$
where
\begin{enumerate}
\item[(a)] $S^\prime\subseteq S_\lambda (f,h)\setminus  S_\lambda (h,g)$ and for every $xH_\lambda \in S^\prime$ we have $E (f,g; xH_\lambda )= E (f,h; xH_\lambda)$.
\item[(b)] $S^{\prime\prime}\subseteq S_\lambda (h,g)\setminus  S_\lambda (f,h)$ and for every $xH_\lambda \in S^{\prime\prime}$ we have $E(f,g; xH_\lambda )= E (h,g; xH_\lambda).$
\item[(c)] $|F|\le 2$.
\end{enumerate}
\end{lem}

\begin{proof}
First, if $|S_\lambda (f,g)|\leq 2$ the statement is trivial, so we can assume $|S_\lambda (f,g)|>2$. Let $$S_\lambda (f,g)=\{ x_1H_\lambda \preceq x_2H_\lambda \preceq\ldots \preceq x_nH_\lambda\}.$$ We fix any geodesics $q\in \mathcal G(h,g)$ and $r\in \mathcal G(f,h)$. By the first claim of Lemma \ref{sep}, every coset from $S_\lambda (f,g)$ is penetrated by at least one of $q$, $r$. Without loss of generality we may assume that at least one of the cosets from $S_\lambda(f,g)$ is penetrated by $r$. Let $x_iH$ be the largest coset (with respect to the order $\preceq$) that is penetrated by $r$. Thus if $i<n$, then  $x_{i+1}H$ is penetrated by $q$.

Let
$$
S^{\prime} = \{ x_jH_\lambda \mid 1\le j< i\} ,
$$
$$
S^{\prime\prime } = \{ x_jH_\lambda \mid i+1< j\le n\} ,
$$
and
$$
F=S_\lambda (f,g)\setminus (S^\prime\cup S^{\prime\prime}).
$$
Obviously $|F|\le 2$. It remains to prove (a) and (b). We will prove (a) only, the proof of (b) is symmetric.

Fix any $1\le j< i$. Let $p$ be any geodesic from $\mathcal G (f,g)$. By Lemma \ref{sep-ref}, $p$ decomposes as $$ p=p_1a_1p_2a_2p_3,$$ where $a_1$, $a_2$ are $H_\lambda $-components of $p$, $(p_1)_+\in x_jH_\lambda $, and $(p_2)_+\in x_iH_\lambda $. Similarly by the choice of $i$, $r$ decomposes as $$r=r_1br_2,$$ where $b$ is an $H_\lambda $-component of $r$ and $(r_2)_-\in x_iH_\lambda $ (see Fig. \ref{fig2}).

Since $(r_2)_-$ and $(p_2)_+$ belong to the same coset of $H_\lambda$, there exists a path $e$ in $\G$ of length at most $1$ such that $e_-=(p_2)_+$ and $e_+=(r_2)_-$. By Lemma \ref{dist}, we have $\ell (p_1a_1p_2)=\ell (r_1)$.  Hence the path $t=p_1a_1p_2er_2$ has the same length as $r$, i.e., $t\in \mathcal G (f,h)$. Also,
\begin{equation}\label{in}
p_{in}(x_jH_\lambda )=t_{in}(x_jH_\lambda )
\end{equation}
and
\begin{equation}\label{out}
p_{out}(x_jH_\lambda )=t_{out}(x_jH_\lambda ).
\end{equation}

\begin{figure}
 \centering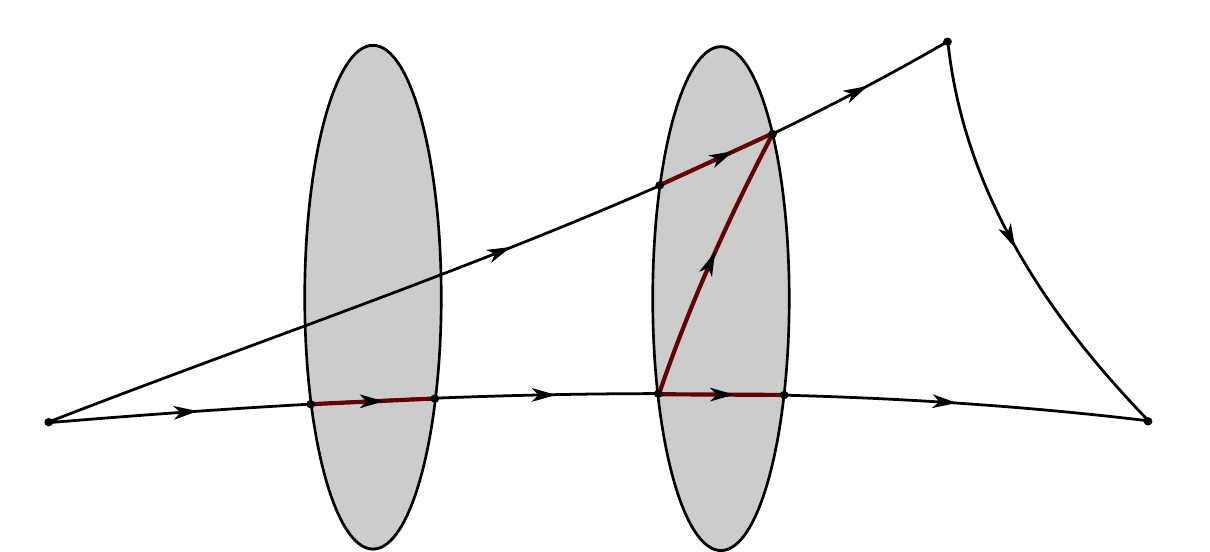\\
  \caption{}\label{fig2}
\end{figure}

So far all our arguments were valid for any $p\in\mathcal G (f,g)$. Since $x_jH_\lambda \in S_\lambda (f,g)$, there exists $p\in\mathcal G (f,g)$ that essentially penetrates $x_jH_\lambda$, i.e., $\dl ((a_1)_-, (a_1)_+)> 3C$ in the above notation. In this case $t$ also essentially penetrates $x_jH_\lambda$. Thus $x_jH\in S_\lambda (f,h)$. Moreover since we have (\ref{in}) and (\ref{out}) for every $p\in \mathcal G(f,g)$, we obtain $E (f,g; x_jH_\lambda )= E (f,h; x_jH_\lambda)$.

To complete the proof of (a) it remains to show that $x_jH_\lambda \notin S_\lambda (h,g)$. Clearly $g\notin x_jH_\lambda$, or $p$ would not be geodesic, so  $x_jH_\lambda $ does not trivially separate $g$ and $h$. Thus, if $x_jH_\lambda \in S_\lambda (h,g)$ there must be a geodesic from $h$ to $g$ which essentially penetrates $x_jH_\lambda$. Hence by Lemma \ref{sep}, every geodesic from $h$ to $g$ penetrates $x_jH_\lambda $, which means $q$ penetrates $x_jH_\lambda $. Then using Lemma \ref{dist}, the fact that every coset of $H_\lambda$ has diameter $1$ with respect to the metric $\dxh$, and the triangle inequality, we obtain
$$
\begin{array}{rl}
\ell(q) &= \dxh (h, x_jH_\lambda) +1 + \dxh (g, x_jH_\lambda) \\ & \\
&> \dxh (h, x_iH_\lambda) +1 + \dxh (g, x_iH_\lambda) \\ & \\
&\geq \ell (r_2) + \dxh ((r_2)_-, (p_3)_-) + \ell (p_3) \\ & \\
&\ge  \dxh (h,g).
\end{array}
$$
Since one of the inequalities is strict, this contradicts the assumption that $q$ is geodesic.
\end{proof}

\section{Extending quasi-cocycles}\label{ext}

We keep all assumptions and notation from the previous section. For each $\lambda \in \Lambda $, let $$\mathcal F_{\lambda}=\{h\in H_{\lambda}\;|\;h\in H_{\mu} \text{ for some } \mu\neq\lambda\}.$$ In particular, if $\Hl$ consists of a single subgroup $H$, the corresponding subset $\mathcal F=\emptyset$.

It follows from Lemma \ref{C} that every $h\in F_{\lambda}$ satisfies $\dl(1, h)\leq 2C$, where $C$ is the constant from Lemma \ref{C}. Indeed for every such $h$ there is a loop $e_1e_2$ in $\G $, where $e_1$ is an edge labeled by $h\in H_\lambda \setminus\{1\}$ and $e_2$ is an edge labeled by the copy of $h$ in $H_\mu\setminus\{1\} $ for some $\mu\in \Lambda $. Since the metric space $(H_\lambda , \dl)$ is locally finite by the definition of a hyperbolically embedded collection of subgroups, we obtain the following.

\begin{lem}\label{Ffinite}
$|\mathcal F_{\lambda}|<\infty$ for all $\lambda\in\Lambda$.
\end{lem}

Also, for $q_\lambda\in\QZ (H_\lambda, U_\lambda)$, let
\begin{equation}\label{Kl}
K_\lambda =\max\{ \| q_\lambda (g)\| : \dl (1,g)<15C\}.
\end{equation}
Observe that the constant $K_\lambda $ is well-defined by local finiteness of $(H_\lambda , \dl)$.

We can now state our main extension theorem in its full generality. Recall that for a quasi-cocycle $q$, $D(q)$ denotes its defect defined by (\ref{Dq}).

\begin{thm}\label{real}
Let $G$ be a group, $\Hl$ a hyperbolically embedded collection of subgroups of $G$, $V$ a normed $G$-module. For each $\lambda \in \Lambda $, let $U_\lambda$ be an $H_\lambda $-submodule of $G$.  Then there exists a linear map $$\iota \colon \bigoplus\limits_{\lambda\in \Lambda }\QZas (H_\lambda, U_\lambda)\to \QZas (G,V)$$ such that for any $q=(q_\lambda)_{\lambda \in \Lambda} \in \bigoplus\limits_{\lambda\in \Lambda }\QZas (H_\lambda, U_\lambda)$  the following hold.
\begin{enumerate}
\item[(a)] For any $\lambda\in \Lambda $ and any $h\in H_{\lambda}\setminus\mathcal F_{\lambda}$, we have  $\iota (q)(h)= q_{\lambda}(h).$ In particular, $\sup\limits_{h\in H_\lambda} \| \iota(q)(h)-q_\lambda (h)\| <\infty$.
\item[(b)] $D(\iota (q))\le  \sum_{\lambda}(54
K_\lambda  + 66 D (q_\lambda)).$
\end{enumerate}
\end{thm}

Notice that the sum in part (b) is finite because $q_\lambda\equiv 0$ for all but finitely many $\lambda$, and thus $K_\lambda=D(q_\lambda)=0$ for all but finitely many $\lambda$. If $G$ contains a singe hyperbolically embedded subgroup, Theorem \ref{real} obviously reduces to Theorem \ref{ind} mentioned in the introduction. Using Lemma \ref{antisym}, one can also obtain a general version of Corollary \ref{cor1}. We leave this to the reader.

Throughout the rest of the section, we use the notation of Theorem \ref{real}. Although our proof can be entirely written in the language of quasi-cocycles, the following concept helps making some arguments more symmetric and easier to comprehend. In the definition below, we write $s(a)=t(a)$ for two partial maps $s,t\colon A\to B$ if the value $s(a)$ is defined if and only if $t(a)$ is, and these values are equal whenever defined.

\begin{defn}\label{bc}
A \emph{partial bi-combing of $G$ with coefficients in $V$} is a partial map  $r\colon G\times G\to V$. We say that
\begin{enumerate}
\item[(a)] $r$ is \emph{$G$-equivariant} if $h r (f,g)=r (hf, hg)$ for any $f,g,h\in G$;
\item[(b)] $r$ is \emph{anti-symmetric} if $r(f,g)=-r(g,f)$ for any $f,g\in G$.
\item[(c)] $r$ has \emph{bounded area} if there exists a constant $A$ such that for any $f,g,h\in G$  for which $r(f,g)$, $r(g,h)$, and $r(h,f)$ are defined, we have
\begin{equation}\label{ba}
\| r(f,g)+r(g,h)+r(h,f)\| \le A.
\end{equation}
The infimum of all $A$ satisfying (\ref{ba}) is called the {\it area} of $r$ and is denoted by $A(r)$.
\end{enumerate}
\end{defn}

Let us fix $\lambda \in \Lambda $. Given $q_\lambda\in \QZas (H_\lambda, U_\lambda)$, we define a partial map $r_\lambda \colon G\times G\to V$ by
$$
r_\lambda (f,g)=fq_\lambda (f^{-1}g).
$$
Thus $r_\lambda (f,g)$ is defined if and only if $f$ and $g$ belong to the same coset $xH_\lambda$.

\begin{lem}\label{r}
The partial map $r_\lambda \colon G\times G \to V$ is an anti-symmetric equivariant partial bi-combing of $G$  of area
\begin{equation}\label{ArDq}
A(r_\lambda)\le  D(q_\lambda).
\end{equation}
\end{lem}

\begin{proof}
Equivariance of $r_\lambda$ is obvious and anti-symmetry follows immediately from anti-symmetry of $q_\lambda$. By equivariance it suffices to verify the bounded area condition for the a triple $1, g,h\in G$. We have
$$
\| r_\lambda(1,g)+r_\lambda(g,h)+r_\lambda(h,1)\| = \| q_\lambda (g) +gq_\lambda(g^{-1}h) -q_\lambda(h)\|  \le  D(q_\lambda).
$$
\end{proof}

\begin{cor}\label{bcpoly}
For any $n\in \mathbb N$, any $x\in G$, and any $g_0, \ldots, g_n\in xH_\lambda$, we have
$$
\left\| r_\lambda(g_0, g_n)- \sum\limits_{i=1}^n r_\lambda(g_{i-1}, g_i)\right\| \le (n-1)D(q_\lambda).
$$
\end{cor}
\begin{proof}
The claim follows from anti-symmetry, the definition of area, and (\ref{ArDq}) by induction.
\end{proof}

Our next goal is to construct a globally defined anti-symmetric  bounded area $G$-equivariant bi-combing  $\widetilde r_{\lambda}\colon G\times G\to V$ that extends $r_\lambda $. To this end, for each  $f,g\in G$ and each coset $xH_\lambda$, we define the average
$$
R_{av}(f,g;xH_{\lambda})= \frac{1}{|E(f,g;xH_{\lambda})|} \sum\limits_{(u, v)\in E(f,g; xH_{\lambda})}r_\lambda (u,v).
$$
If $xH_\lambda \notin S_\lambda (f,g)$, we assume $R_{av}(f,g;xH_{\lambda})= 0$.
Note that $R_{av}(f,g;xH_{\lambda})$ is well-defined since $E(f,g;xH_{\lambda})<\infty $ by part (c) of Lemma \ref{Efg}.

\begin{lem}\label{Eavfg}
For any $f,g,h,x\in G$, the following hold.
\begin{enumerate}
\item[(a)] $R_{av}(f,g;xH_{\lambda})=-R_{av}(g,f;xH_{\lambda})$.
\item[(b)] $R_{av}(hf,hg;hxH_{\lambda})=R_{av}(f,g;xH_{\lambda})$.
\item[(c)] For any $(u,v)\in E(f,g; xH_\lambda)$, we have
\begin{equation}\label{rR}
\|r_\lambda(u,v)-R_{av}(f, g;xH_{\lambda})\|\leq 2D(q_\lambda)+2K_\lambda.
\end{equation}
\end{enumerate}
\end{lem}
\begin{proof}
The first claim follows from parts (a) of Lemma \ref{Efg} and anti-symmetry of $r_\lambda $. The second claim follows from parts (b) of Lemma \ref{Efg} and the equivariance  of $r_\lambda$.

To prove (c), note that for any $(u^\prime, v^\prime)\in E(f, g;xH_{\lambda})$, we have $$\max\{\dl(u, u^\prime), \dl(v, v^\prime)\}\leq 3C$$ by Lemma \ref{sep}. Thus, using the triangle inequality and applying Corollary \ref{bcpoly} to elements $u,u^\prime, v^\prime, v\in xH_\lambda $, we obtain
$$\begin{array}{rl}
\|r_{\lambda}(u, v)-r_\lambda( u^\prime,  v^\prime)\| \le &
 \|r_{\lambda}( u, v)-r_{\lambda}( u,  u^\prime)-r_\lambda( u^\prime, v^\prime)-r_{\lambda}( v^\prime, v)\| \\ & \\ +&\|r_{\lambda}( u,  u^\prime)\|+\|r_{\lambda}( v^\prime,  v)\|
\leq 2D(q_\lambda)+2K_\lambda.
\end{array}$$
This obviously implies (\ref{rR}).
\end{proof}

Let
\[
\widetilde r_{\lambda}(f,g)=\sum\limits_{xH_{\lambda}\in S_{\lambda}(f,g)}R_{av}(f,g;xH_{\lambda}).
\]
Note that $\widetilde r_\lambda$ is well-defined as $S_\lambda (f,g)$ is finite for any $f,g\in G$ by Corollary \ref{Sfin}.

\begin{lem}\label{risbc}
The map $\widetilde r_{\lambda}\colon G\times G\to V$ is an anti-symmetric $G$-equivariant bi-combing of area \begin{equation}\label{A-r-tilde}
 A(\widetilde r_\lambda )\le 66D(q_\lambda )+54 K_\lambda .
\end{equation}
\end{lem}
\begin{proof}
Equivariance and anti-symmetry of $\widetilde r_\lambda$ follow immediately from Lemma \ref{Sfg} and Lemma \ref{Eavfg}. In order to show that $\widetilde{r}_{\lambda}$ satisfies the bounded area condition, we need to estimate the norm of $\widetilde r_\lambda(f,g)+\widetilde r_\lambda (g,h)+ \widetilde r_\lambda (h,f)$
uniformly on $f,g,h\in G$. Since $R_{av}(f,g; xH_\lambda)=0$ if $xH_\lambda \notin S_\lambda(f,g)$, we have
$$
\widetilde r_\lambda(f,g)+\widetilde r_\lambda (g,h)+ \widetilde r_\lambda (h,f)  =\sum_{xH_\lambda \in G/H_\lambda} \rho (f,g,h; xH_\lambda),
$$
where
$$
\rho (f,g,h; xH_\lambda)\colon  = R_{av}(f,g; xH_\lambda) +R_{av}(g,h; xH_\lambda)+R_{av}(h,f; xH_\lambda).
$$
Of course, $\rho (f,g,h; xH_\lambda)$ is nontrivial only if $xH_\lambda \in S_\lambda (f,g)\cup S_\lambda(g,h)\cup S_\lambda(h,f)$.

Fix $f,g,h\in G$. We start by estimating $\rho (f,g,h; xH_\lambda)$ for cosets from $S_\lambda(f,g)$.
Let $S_{\lambda}(f,g)=S^{\prime}\sqcup S^{\prime\prime}\sqcup F$ be the decomposition provided by Lemma \ref{fgh}. Suppose first that $xH_{\lambda}\in S^{\prime}$. Then $xH_{\lambda}\in S_{\lambda}(f,h)=S_{\lambda}(h,f)$ and $E(f,g;xH_{\lambda})=E(f,h;xH_{\lambda})$ by Lemma \ref{fgh}. Hence
\begin{equation}\label{contr1}
R_{av}(f,g; xH_\lambda)= R_{av}(f,h; xH_\lambda)=-R_{av}(h,f; xH_\lambda).
\end{equation}
by Lemma \ref{Eavfg} (a). On the other hand, Lemma \ref{fgh} also states that $xH_\lambda \notin S_\lambda (h,g)=S_{\lambda}(g,h)$. Hence
\begin{equation}\label{contr2}
R_{av}(g,h; xH_\lambda)= 0.
\end{equation}
Summing up (\ref{contr1}) and (\ref{contr2}), we obtain $\rho (f,g,h; xH_\lambda)=0$. Similarly, $\rho (f,g,h; xH_\lambda)=0$ for any $xH_\lambda \in S^{\prime\prime}$. Thus
\begin{equation}\label{redtoF}
\sum\limits_{xH_\lambda \in S_\lambda (f,g)} \rho (f,g,h; xH_\lambda ) =\sum\limits_{xH_\lambda \in F} \rho (f,g,h; xH_\lambda ).
\end{equation}

Fix a coset $xH_{\lambda}\in F$ and any $p\in \mathcal G(f,g)$, $q\in \mathcal G(h,g)$, $r\in \mathcal G(f,h)$. There are three cases to consider.

\begin{figure}
 \centering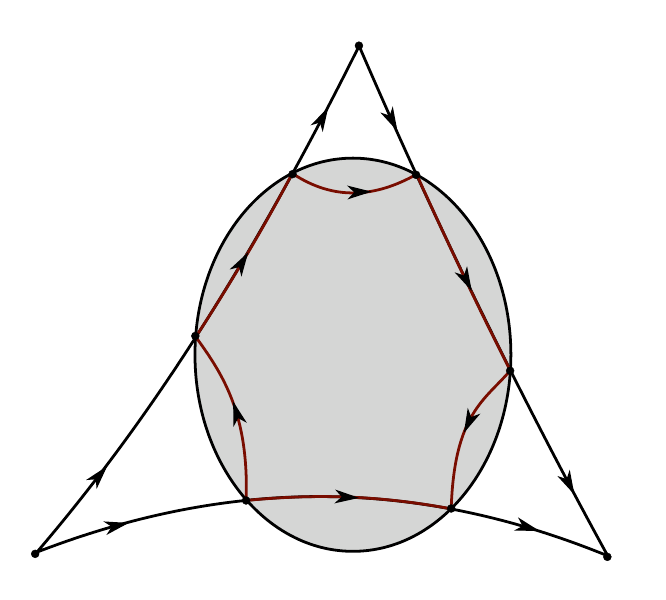\\
  \caption{}\label{fig4}
\end{figure}

\paragraph{\bf Case 1:} $xH_{\lambda}\in S_{\lambda}(g, h)\cap S_{\lambda}(h, f)$.
In this case we have $p=p_1ap_2$, $q=q_1cq_2$, $r=r_1br_2$, where $a$, $c$, and $b$ are $H_\lambda$-components of $p$, $q$, and $r$, respectively, corresponding to the coset $xH_\lambda$ (i.e., $a_{\pm}, b_{\pm}, c_{\pm}\in xH_\lambda$). Let $e_1$, $e_2$, $e_3$ be paths of lengths at most $1$ labeled by elements of $H_\lambda$ and connecting $a_-$ to $b_-$, $b_+$ to $c_-$, and $c_+$ to $a_+$ (see Fig. \ref{fig4}).

Since a geodesic in $\G$ can penetrate a coset of $H_\lambda $ at most once, $e_1$ is either trivial or is an isolated component of a geodesic triangle (namely $p_1e_1r_1^{-1}$). The same holds true for $e_1$ and $e_2$. Hence by Lemma \ref{C}, we obtain
\begin{equation}\label{ei}
\dl ((e_i)_-, (e_i)_+)\le 3C,\;\;\; i=1,2,3.
\end{equation}
In particular,
\begin{equation}\label{rei}
\| r_\lambda ((e_i)_-, (e_i)_+)\| \le K_\lambda ,\;\;\; i=1,2,3.
\end{equation}
by the definition of $K_\lambda$ (see (\ref{Kl})).
Using the triangle inequality, applying Lemma \ref{bcpoly} to the vertices of the hexagon $e_1be_2ce_3a^{-1}$, and using (\ref{rei}), we obtain
\begin{multline*}
\| r_\lambda (a_-, a_+) +r_\lambda (b_+,b_-) + r_\lambda (c_+,c_-)\| \\ \\\le \left\| r_\lambda (a_-, a_+) -r_\lambda (b_-,b_+) - r_\lambda (c_-,c_+)- \sum\limits_{i=1}^3 r_\lambda ((e_i)_-, (e_i)_+) \right\| +\left\| \sum\limits_{i=1}^3 r_\lambda ((e_i)_-, (e_i)_+)\right\| \\ \\\le 5D(q_\lambda) + 3K_\lambda.
\end{multline*}
Now Lemma \ref{Eavfg} (c) implies
\begin{equation}\label{r-c1}
\begin{array}{rl}
\|\rho (f,g,h; xH_\lambda)\|&= \|R_{av}(f,g; xH_\lambda) +R_{av}(g,h; xH_\lambda)+R_{av}(h,f; xH_\lambda) \|\\&\\ &\leq  \| r_\lambda (a_-, a_+) + r_\lambda (c_+,c_-) +r_\lambda (b_+,b_-)\| + 6(D(q_\lambda)+K_\lambda)\\ & \\
& \leq  11D(q_\lambda)+9K_\lambda.
\end{array}
\end{equation}

\paragraph{\bf Case 2:} $xH_{\lambda}\in S_{\lambda}(h, f)\setminus S_{\lambda}(g, h)$ or $xH_{\lambda}\in S_{\lambda}(g, h)\setminus S_{\lambda}(h, f)$. Since the proof in these cases is the same, we will only consider the case $xH_{\lambda}\in S_{\lambda}(h, f)\setminus S_{\lambda}(g, h)$.  Let $p=p_1ap_2$, $r=r_1br_2$, and $e_1$ be as in Case 1 and let $e$ be the path of length at most $1$ in $\G$ connecting
$b_+$ to $a_+$ and labeled by an element of $H_\lambda$. There are two possibilities to consider.

{\bf 2a)} First assume that $e$ is isolated in the quadrilateral $ep_2q^{-1}r_2^{-1}$ (see Fig. \ref{fig5}). Then we have $\dl (e_-,e_+)\le 4C$ by Lemma \ref{C} and hence
$$
\|r_\lambda(e_-,e_+)\| \le K_\lambda.
$$
Note that (\ref{rei}) remains valid for $i=1$. Applying Corollary \ref{bcpoly} to the vertices of the quadrilateral $e_1bea^{-1}$ as in Case 1 we obtain
$$
\begin{array}{rl}
\| r_\lambda (a_-, a_+) +r_\lambda (b_+,b_-) \|  \le &\| r_\lambda (a_-, a_+) -r_\lambda ((e_1)_-,(e_1)_+) - r_\lambda (b_-,b_+)-  r_\lambda (e_-, e_+) \| \\ &\\&  +\| r_\lambda ((e_1)_-, (e_1)_+)\| + \| r_\lambda (e_-,e_+)\| \le 3D(q_\lambda) + 2K_\lambda.
\end{array}
$$
Since $xH_\lambda \notin S_\lambda (g,h)$, we have $R_{av}(g,h;xH_\lambda )=0$. Finally Lemma \ref{Eavfg} (c) implies
\begin{equation}\label{r-c2a}
\begin{array}{rl}
\|\rho (f,g,h; xH_\lambda)\|& =\| R_{av}(f,g; xH_\lambda) +R_{av}(h,f; xH_\lambda)\| \\ &\\
&\leq  \| r_\lambda (a_-, a_+) + r_\lambda (b_+,b_-)\| + 4(D(q_\lambda)+K_\lambda )\\&\\&
\leq  7D(q_\lambda)+6K_\lambda.
\end{array}
\end{equation}

{\bf 2b)} Suppose now that $e$ is not isolated in the quadrilateral $ep_2q^{-1}r_2^{-1}$. Then $e$ is connected to a component $c$ of $q$. Let $q=q_1cq_2$ and let $e_1$ and $e_2$ be as in Case 1 (see Fig. \ref{fig4}). Then (\ref{rei}) remains valid. In addition, we have $\dl (c_-,c_+) \le 3C$ as $xH_\lambda \notin S_\lambda(g,h)$ and hence $q$ can not essentially penetrate $xH_\lambda$. Hence $\| r_\lambda (c_-,c_+)\| \le K_\lambda $.
The reader can easily verify that arguing as in the Case 1 and then as in (\ref{r-c2a}), we can obtain
$$
\| r_\lambda (a_-, a_+) + r_\lambda (b_+,b_-)\| \le 5A_\lambda +4 K_\lambda
$$
and consequently
\begin{equation}\label{r-c2b}
\|\rho (f,g,h; xH_\lambda)\| \le 9D(q_\lambda) + 8K_\lambda .
\end{equation}

\begin{figure}
 \centering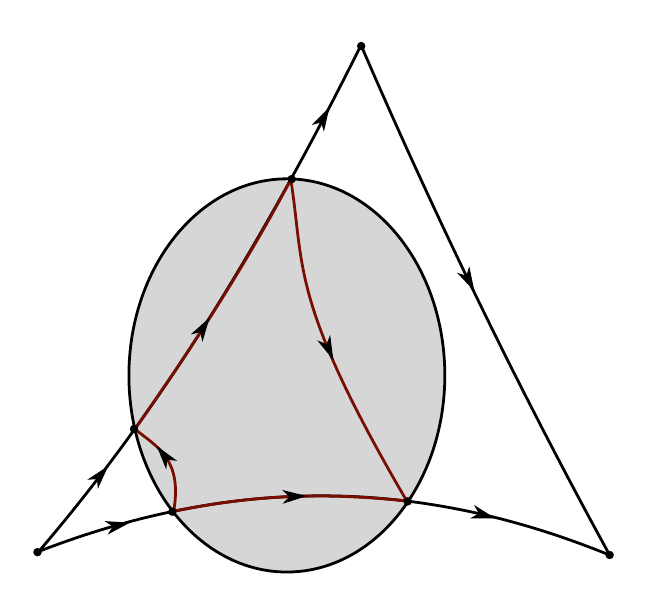\\
  \caption{}\label{fig5}
\end{figure}

\paragraph {\bf Case 3:} $xH_{\lambda}\notin S_{\lambda}(h, f)\cup S_{\lambda}(g, h)$. Let $p=p_1ap_2$ be as in Cases 1 and 2. There are three possibilities to consider.

{\bf 3a)} $a$ is an isolated component of $pq^{-1}r^{-1}$. In this case $\dl (a_-,a_+)\le 3C$.

{\bf 3b)} $a$ is connected to a component of exactly one of $q$, $r$. For definiteness, assume that $a$ is connected to a component $b$ of $r$. Then, in the notation of Case 2 (see Fig. \ref{fig5}), $e$ is isolated in $ep_2q^{-1}r_2^{-1}$ and we have $\dl (e_-,e_+)\le 4C$ by Lemma \ref{C}. As in Case 1, we have  (\ref{ei}) for $i=1$. Since $xH_{\lambda}\notin S_{\lambda}(h, f)$, $r $ can not essentially penetrate $xH_\lambda$. Thus $\dl (b_-,b_+)\le 3C$. Applying the triangle inequality to the quadrilateral $e_1bea^{-1}$, we obtain
$$
\dl (a_-,a_+)\le 10C.
$$

{\bf 3c)} $a$ is connected to a component $b$ of $r$ and a component $c$ of $q$. Then in the notation of Case 1 and Fig. \ref{fig4}, inequalities (\ref{ei}) remain valid and we also have $\dl (b_-, b_+)\le 3C$ and $\dl (c_-,c_+)\le 3C$ as in Case 3b). Applying the triangle inequality to the hexagon $e_1be_2ce_3a^{-1}$, we obtain
$$
\dl (a_-,a_+)\le 15C.
$$

Thus, in all cases 3a) - 3c) we have $\| r_\lambda (a_-,a_+)\| \le K_\lambda $. Since $R_{av} (g,h;xH_\lambda)=R_{av} (h,f;xH_\lambda)=0$ in this case, using Lemma \ref{Eavfg} (c) we obtain
\begin{equation}\label{r-c3}
\| \rho (f,g,h; xH_\lambda)\| = \| R_{av} (f,g;xH_\lambda)\| \le 2A(r_\lambda ) + 3K_\lambda.
\end{equation}
in Case 3.

Summarizing (\ref{redtoF}), (\ref{r-c1}), (\ref{r-c2a}), (\ref{r-c2b}), (\ref{r-c3}), and taking into account that $|F|\le 2$, we obtain
$$
\left\| \sum\limits_{xH_\lambda \in S_\lambda (f,g)} \rho (f,g,h; xH_\lambda )\right\| = \left\| \sum\limits_{xH_\lambda \in F} \rho (f,g,h; xH_\lambda )\right\| \le 22 D(q_\lambda) +18K_\lambda.
$$

Repeating the same arguments for $S_{\lambda}(h,f)$ and $S_{\lambda}(g, f)$ and summing up, we obtain
$$
\left\| \widetilde r_\lambda(f,g)+\widetilde r_\lambda (g,h)+ \widetilde r_\lambda (h,f)\right\| \leq 66D(q_\lambda)+54K_\lambda.
$$
\end{proof}

We are now ready to prove the main extension theorem.

\begin{proof}[Proof of Theorem \ref{real}]
Let $q=(q_\lambda)_{\lambda \in \Lambda} \in \bigoplus\limits_{\lambda\in \Lambda }\QZas(H_\lambda, U_\lambda)$. For each $\lambda\in \Lambda$, let $\widetilde{r}_\lambda$ be the bi-combing constructed above and let $\widetilde{q}_\lambda(g)=\widetilde r_\lambda(1,g)$. Then $\widetilde q_\lambda \in \QZas(G,V)$. Indeed we have
\begin{equation}\label{qc}
\begin{array}{rl}
\|\widetilde{q}_\lambda(fg)-\widetilde{q}_\lambda(f)-f\widetilde{q}_\lambda(g)\|  & = \| \widetilde r_\lambda(1,fg)-\widetilde r_\lambda(1,f)-f\widetilde r_\lambda(1,g)\| \\&\\
& =\| \widetilde r_\lambda(1,fg) + \widetilde r_\lambda(f,1) + \widetilde r_\lambda(fg, f)\| \\&\\ &\le A(\widetilde r_\lambda).
\end{array}
\end{equation}
Anti-symmetry of $\widetilde q_\lambda$ follows from that of $\widetilde r_\lambda$.

Further we define
\[
\iota(q)=\sum_{\lambda\in \Lambda}\widetilde{q}_{\lambda}.
\]
Since $q$ is supported on only finitely many $\lambda$, $\iota(q)$ is equal to a finite linear combination of quasi-cocycles, so $\iota(q)\in \QZas(G, V)$. It is easy to see that the maps $\QZas(H_\lambda, U_\lambda)\to \QZas(G,V)$ defined by $q_\lambda\mapsto \widetilde{q}_\lambda$ are linear. Hence so is $\iota$.

If $h\in H_{\lambda}\setminus\mathcal F_{\lambda}$, then $S_\lambda (1,h)=\{ H_{\lambda}\}$ and $S_\mu (1,h)=\emptyset $ for any $\mu\ne \lambda $.  Obviously $E(1,h; H_\lambda)=\{ (1,h)\}$. Thus $\widetilde{r}_\lambda(1,h)=r_\lambda (1,h)=q_\lambda(h)$ and $\widetilde{r}_\mu (1,h)=0$ whenever $\mu\ne \lambda $. Thus
$$
\iota(q)(h)=\sum_{\mu\in \Lambda} \widetilde{q}_{\mu}(h)=\sum_{\mu\in \Lambda}\widetilde{r}_\mu(1, h)=q_{\lambda}(h).
$$
This finishes the proof of (a). Part (b) follows from (\ref{qc}) and (\ref{A-r-tilde}).
\end{proof}

\begin{rem}\label{as-nec}
Our proof essentially uses the fact that the quasi-cocycles $q_\lambda$ are anti-symmetric. In fact, our approach provably fails for non-anti-symmetric ones. This can be illustrated in the case when $G=F(x,y)$, the free group of rank $2$, and $H=\langle x\rangle$. Indeed take $q\in \QZ (H, \mathbb R)$ defined by
$$
q(x^n)=\left\{\begin{array}{ll}
1, & {\rm \; if \;} n\ge 0,\\
0, & {\rm \; if \;} n<0.
\end{array}\right.
$$
Let $\widetilde q$ be the extension obtained as above using the subset $X=\{ x,y\}$ of $G$. Take any $n\in \mathbb N$ such that $\widehat d (1,x^n)> 3C$ (in fact, $C=0$ in this case, but we will not use this). Then it is straightforward to verify that $\widetilde q((yx^n)^k) =k$ while $\widetilde q((yx^n)^{-k})=\widetilde q((x^{-n}y^{-1})^k)=0$ for every $k\in \mathbb N$. This contradicts the quasi-cocycle identity as $k\to \infty $. A similar argument shows that the anti-symmetry condition can not be dropped in Example \ref{ex2}.
\end{rem}

\section{Applications}\label{app}

\paragraph{Bounded cohomology.} Our goal here is to prove Corollary \ref{cor1}. We begin with an auxiliary result.

\begin{prop}\label{dim}
Let $G$ be a group, $H$ a \he subgroup of $G$, $V$ a $G$-module, and $U$ an $H$-submodule of $V$. Suppose that there exists a continuous projection $\pi \colon V\to U$. Then there is a linear map $\phi\colon \QZ(H,U)\to EH_b^2(G,V)$ such that $\Ker \phi \subseteq \ell^{\infty}(H,U)+ Z^1(H,U)$. In particular, $$\dim H^2_b(G,V)\ge \dim EH_b^2(G,V) \ge \dim EH_b^2(H,U).$$
\end{prop}

\begin{proof}
We define $\phi$ to be the composition $\delta\circ\kappa$, where $\kappa$ is given by Corollary \ref{ind1} and $\delta $ is the natural map $\QZ(G,V)\to EH_b^2(G,V)$ (see Lemma \ref{seq}). Note that if $\phi (q)=0$ for some $q\in \QZ(H,U)$, then
\begin{equation}\label{b+h}
\kappa (q)=h+b,
\end{equation}
where  $b\in \ell^{\infty}(G,V)$ and $h\in Z^1(G,V)$. Since $\kappa(q)(x)\in U$ for all $x\in H$, composing both sides of this equality with $\pi$ and restricting to $H$ we obtain  $$\kappa (q)|_H=\pi\circ h\vert_H +\pi\circ b\vert_H.$$ Obviously $\pi\circ h\vert_H\in Z^1(H,U)$ and $\pi\circ b\vert_H\in \ell^\infty (H,U)$ since $\pi $ is continuous. By Corollary \ref{ind1}, $(q-\kappa (q)|_H)\in \ell^{\infty}(H,U)$, thus $q\in \ell^{\infty}(H,U)+ Z^1(H,U)$.
\end{proof}

The next lemma is a simplification of \cite[Theorem 2.23]{DGO}.

\begin{lem}\label{vf}
Let $G\in \mathcal X$. Then for every $n\in \mathbb N$, there exists a subgroup $H_n\h G$ such that $H_n\cong F_n\times K$, where $F_n$ is the free group of rank $n$ and $K$ is finite.
\end{lem}

We are now ready to prove Corollary \ref{cor1}.

\begin{proof}[Proof of Corollary \ref{cor1}]
It is easy to see that the assumptions of Lemma \ref{dim} hold in the case $V=\ell^p (G)$ and $U=\ell^p(H)$. It is well known that $\dim EH_b^2(H)=\infty $ for every virtually free group which is not virtually cyclic (see. e.g., \cite{Ham1}). To complete the proof it remains to note that every group $G\in \mathcal X$ contains a virtually free but not virtually cyclic \he subgroup by Lemma \ref{vf}.
\end{proof}

\paragraph{Stable commutator length.}
Let $G$ be a group, and let $g\in [G,G]$. The the \emph{commutator length} of $g$, denoted $cl_G(g)$, is defined as the minimal number of commutators whose product is equal to $g$ in $G$. The \emph{stable commutator length} is defined by
\[
scl_G(g)=\lim_{n\rightarrow \infty}\frac{cl_G(g^n)}{n}.
\]
It is customary to extend $scl_G$ to all elements $g$ for which have some positive power $g^n\in[G,G]$ by letting $scl_G(g)=\frac{scl(g^n)}{n}$. Basic facts and theorems about stable commutator length can be found in \cite{DC}.

Following \cite{DC}, we will denote space of \emph{quasimorphisms} on $G$ by $\widehat Q(G)$. Recall that this is the same as  $\QZ(G, \mathbb R)$ where $\mathbb R$ is considered as a $G$-module with the trivial action. Note that  in this setting Theorem \ref{ind} says that any anti-symmetric quasimorphism on $H$ can be extended to a quasimorphism on $G$.

 A quasimorphism $\varphi$ on $G$ is called \emph{homogeneous} if for all $g\in G$ and all $n\in \mathbb Z$, $\varphi(g^n)=n\varphi(g)$. In particular, all homogeneous quasimorphisms are anti-symmetric. We denote the subspace of homogeneous quasimorphisms by $Q(G)$. The connection between quasimorphisms and stable commutator length is provided by the Bavard Duality Theorem \cite{B}.

\begin{thm}[Bavard Duality Theorem]
For any $g\in[G,G]$, there is an equality
\begin{equation}\label{bd}
scl_G(g)=\sup_{\varphi\in Q(G)}\frac{\varphi(g)}{2D(\varphi)}.
\end{equation}
Where the supremum is taken over all homogeneous quasimorphisms of non-zero defect.
\end{thm}

In fact, it is not hard to see that this supremum is always realized by some quasimorphism.

Given any quasimorphism $\varphi$, there is a standard way to obtain a homogeneous quasimorphism $\psi$, called the homogenization of $\varphi$. This is done by defining
\[
\psi(g)=\lim_{n\rightarrow \infty}\frac{\varphi(g^n)}{n}.
\]

\begin{lem}[{\cite[Corollary 2.59]{DC}}]\label{homo}
Let $\varphi\in \widehat Q(G)$ with homogenization $\psi$. Then $D(\psi)\leq 2D(\varphi)$.
\end{lem}

Our plan for proving Corollary \ref{scl} will be to take an element $h\in H$ and apply Bavard Duality to find a homogeneous quasimorphism which which realizes (\ref{bd}) with respect to $scl_H$. Then we can use Theorem \ref{ind} to extend this to a quasimorphism on all of $G$, then apply Bavard Duality again to find a lower bound on $scl_G(h)$. In order to do this we will need to understand the defect of the extended quasimorphism.

Let $H$ be a group, and let $\xi\colon H \to H/[H,H]\otimes \mathbb Q$ be the natural map. A subset $Y\subseteq H$ will be called \emph{nice} if $Y$ can be decomposed as $Y=Y_1\cup Y_2$ such that $\xi(Y_1)$ is linearly independent and $\xi|_{Y_2}\equiv 0$.

\begin{lem}\label{nice}
Every finitely generated subgroup of $H$ has a nice finite generating set.
\end{lem}
\begin{proof}
Let $H^\prime$ be a finitely generated subgroup of $H$, and let $X$ be a finite generating set of $H^\prime$. Then $\xi(H^\prime)$ is a finitely generated subgroup of a torsion-free abelian group, and hence $\xi(H^\prime)$ is a finitely generated free abelian group. Let $\{v_1,...,v_n\}$ be a basis for $\xi(H^\prime)$ as a free abelian group and let $y_i\in H^\prime$ be such that $\xi(y_i)=v_i$. Then for each $x\in X$, there exist integers $a_{x, 1},...a_{x,n}$ such that $\xi(x)=\sum\limits_{i=1}^n a_{x,i}v_i$. Let $\hat{x}=xy_1^{-a_{x,1}}... y_n^{-a_{x,n}}$. Now let $Y_1=\{y_1,...y_n\}$, and let $Y_2=\{\hat{x} \;|\; x\in X\}$. Then clearly $Y=Y_1\cup Y_2$ is nice, and $\langle Y\rangle=\langle X\rangle=H^\prime$.

\end{proof}

\begin{lem}[{\cite[Theorem 16.1]{F}}]\label{abel}
Let $B$ be a subgroup of an abelian group $A$, and let $D$ be a divisible abelian group. Then every homomorphism from $B\to D$ can be extended to a homomorphism from $A\to D$.
\end{lem}

The reason we are interested in nice subsets is the following lemma.
\begin{lem}\label{mq}
For any group $H$, any nice finite subset $Y\subseteq H$, and any $\varphi\in Q(H)$, there exists $\varphi^\prime\in Q(H)$ such that $\varphi^{\prime} |_{[H, H]}\equiv \varphi |_{[H, H]}$, $D(\varphi^{\prime})=D(\varphi)$, and for all $y\in Y$,
\[
|\varphi^{\prime}(y)|\leq 2D(\varphi^{\prime})scl_H(y).
\]
\end{lem}
\begin{proof}
Let $Y=Y_1\cup Y_2$ be the decomposition given by the definition of a nice subset. If $y\in Y_2$, then there exists some $n$ such that $y^n\in [H, H]$. Then for any $\varphi\in Q(H)$, Bavard Duality gives
\begin{equation}\label{Kbound}
|\varphi(y)|\leq 2D(\varphi)scl_H(y).
\end{equation}
Now, let $A=H/[H,H]$, and let $B$ be the image of $\langle Y\rangle$ inside $A$. Let $\theta$ be the quotient map $\theta\colon H\to A$. Then by definition of nice subsets we can define a homomorphism $\alpha:B\to\mathbb R$ such that $\alpha(\theta(y))=\varphi(y)$ for all $y\in Y_1$. Since $\mathbb R$ is divisible, Lemma \ref{abel} allows us to extend $\alpha$ to all of $A$. Composing $\alpha$ with $\theta$ gives a homomorphism $\beta:H\to\mathbb R$ which satisfies $\beta(y)=\varphi(y)$ for all $y\in Y_1$. Now we set $\varphi^{\prime}=\varphi-\beta$. Since $\beta$ vanishes on $[H,H]$, $\varphi^{\prime} |_{[H, H]}\equiv \varphi |_{[H, H]}$. Since $\varphi^{\prime}$ is a shift of $\varphi$ by a homomorphism, $D(\varphi^{\prime})=D(\varphi)$. Furthermore, combining the fact that $\varphi^{\prime}(y)=0$ for all $y\in Y_1$ with (\ref{Kbound}), we get that for all $y\in Y$,
\[
|\varphi(y)|\leq 2D(\varphi^\prime)scl_H(y).
\]
\end{proof}

We are now ready to prove Corollary \ref{scl}.

\begin{proof}[Proof of Corollary \ref{scl}]

Let $H\h (G, X)$, and by Lemma \ref{relmet} there exists $Y^{\prime}$ a finite subset of $H$ such that the relative metric $\widehat d$ on $H$ is bi-Lipschitz equivalent to the word metric with respect to $Y^{\prime}$. By Lemma \ref{nice} the subgroup $\langle Y^{\prime}\rangle $ has a nice finite generating set $Y$. Let $d_Y$ be the word metric with respect to $Y$. Then $d_Y$ is bi-Lipschitz equivalent to the relative metric $\widehat d$ on $H$, so there exists a constant $L$ such that for all $f, g\in H$,
\begin{equation}\label{L}
d_Y(f, g)\leq L\widehat{d}(f, g).
\end{equation}

Fix some $h\in [H, H]$, and let $\varphi\in Q(H)$ be the quasimorphism which realizes the Bavard Duality; that is, $scl_H(h)=\frac{\varphi(h)}{2D(\varphi)}$.
Let $\varphi^{\prime}$ be the modified quasimorphism provided by Lemma \ref{mq}.

Let $\iota:Q(H)\to\widehat Q(G)$ be map provided by Theorem \ref{real}. Then by part $(b)$ of Theorem \ref{real} we have
\[
D(\iota(\varphi^\prime))\leq 54K+66D(\varphi^\prime)
\]
where $K$ is defined by $K=\max\{ |\varphi^{\prime}(k)| : \widehat{d} (1,k)<15C\}$. However, by (\ref{L}) we get $K\leq \max\{ |\varphi^{\prime}(k)| : d_Y (1,k)<15CL\}$. Inductively applying the definition of a quasimorphism along with Lemma \ref{mq}, for any such $k$ we get
\[
|\varphi^{\prime}(k)|\leq 15CL(D(\varphi^{\prime})+2D(\varphi^{\prime})\max_{y\in Y}\{scl_H(y)\}).
\]

That is, we have bound $K$  as a constant multiple of $D(\varphi^{\prime})$. Thus there exists a constant $M$ (which is independet of $\varphi^{\prime}$) such that

\begin{equation}\label{B}
D(\iota(\varphi^{\prime}))\leq 54K+66D(\varphi^{\prime})\leq MD(\varphi^{\prime}).
\end{equation}

 Now, $\iota(\varphi^{\prime})$ is a quasimorphism on $G$, and in order to apply Bavard Duality we homogenize $\iota(\varphi^{\prime})$ to get a quasimorphism $\psi$, satisfying $D(\psi)\leq 2D(\iota(\varphi^{\prime}))$. Then applying the definition of $\psi$, along with the homgeneity of $\varphi^{\prime}$ and the conditions of Theorem \ref{ind} gives
\[
\psi(h)=\lim_{n\rightarrow\infty}\frac{\iota(\varphi^{\prime})(h^n)}{n}=\varphi^{\prime}(h)=\varphi(h).
\]
 Also, (\ref{B}) and Lemma \ref{homo} show that $D(\psi)\leq 2D(\iota(\varphi^{\prime}))\leq 2MD(\varphi^{\prime})=2MD(\varphi)$. Applying Bavard Duality again gives
\[
scl_G(h)\geq \frac{\psi(h)}{2D(\psi)}\geq \frac{\varphi(h)}{4MD(\varphi)}=\frac{1}{2M}scl_H(h).
\]

\end{proof}

\begin{proof}[Proof of Corollary \ref{malnorm}]
If $H$ is an almost malnormal quasi-convex subgroup of a hyperbolic group, then $G$ is hyperbolic relative to $H$ \cite{Bow}. Hence $H$ is \he in $G$ by \cite[Proposition 2.4]{DGO} and the claim follows from Corollary \ref{scl}.
\end{proof}

\begin{rem}\label{free-dist}
Note that the malnormality condition can not be dropped in Corollary \ref{malnorm} even for free groups. For example, let $F=F(x,y,t)$ be the free group of rank $3$ with basis $\{x,y,t\}$. In what follows we write $a^b$ for $b^{-1}ab$ and $[a,b]$ for $a^{-1}b^{-1}ab$. Let $H=\langle x,y, x^t, y^t\rangle $ and let $$h_k= [x,y]^{-k}[x^t,y^t]^k.$$

Since the subset $\{ x,y, x^t, y^t\} \subseteq G$ is Nielsen reduced, the subgroup $H$ is freely generated by $x,y, x^t, y^t$. Therefore $scl_H(h_k)= k+1/2$ (see \cite[Example 2.100]{DC}). On the other hand, we have
$$
scl_G(h_k) = scl_G([x,y]^{-k} ([x,y]^k)^t)= scl_G ([[x,y]^k,t])\le 1.
$$
Thus $scl_H(h_k)/scl_G(h_k)\to \infty $ as $k\to \infty $.
\end{rem}

\subsection*{Acknowledgments.} We are grateful to Francois Dahmani and Ionut Chifan, with whom we discussed this project at various stages.

\vspace{1cm}

\noindent {\bf M. Hull: } Department of Mathematics, Vanderbilt University, Nashville TN 37240, USA.\\
E-mail: {\it michael.b.hull@vanderbilt.edu}

\medskip

\noindent {\bf D. Osin: } Department of Mathematics, Vanderbilt University, Nashville TN 37240, USA.\\
E-mail: {\it denis.osin@gmail.com}

\end{document}

%% file: fig0.pdf_tex

\begingroup
  \makeatletter
  \providecommand\color[2][]{%
    \errmessage{(Inkscape) Color is used for the text in Inkscape, but the package 'color.sty' is not loaded}
    \renewcommand\color[2][]{}%
  }
  \providecommand\transparent[1]{%
    \errmessage{(Inkscape) Transparency is used (non-zero) for the text in Inkscape, but the package 'transparent.sty' is not loaded}
    \renewcommand\transparent[1]{}%
  }
  \providecommand\rotatebox[2]{#2}
  \ifx\svgwidth\undefined
    \setlength{\unitlength}{433.25592422pt}
  \else
    \setlength{\unitlength}{\svgwidth}
  \fi
  \global\let\svgwidth\undefined
  \makeatother
  \begin{picture}(1,0.52751731)%
    \put(0,0){\includegraphics[width=\unitlength]{fig0.pdf}}%
    \put(0.30789297,0.25624641){\color[rgb]{0,0,0}\makebox(0,0)[lb]{\smash{$\Gamma_H$}}}%
    \put(0.31115711,0.07051609){\color[rgb]{0,0,0}\makebox(0,0)[lb]{\smash{$x^{-1}\Gamma_H$}}}%
    \put(0.31050429,0.44165028){\color[rgb]{0,0,0}\makebox(0,0)[lb]{\smash{$x\Gamma _H$}}}%
    \put(0.15121363,0.33034268){\color[rgb]{0,0,0}\makebox(0,0)[lb]{\smash{$x$}}}%
    \put(0.20277635,0.33082313){\color[rgb]{0,0,0}\makebox(0,0)[lb]{\smash{$x$}}}%
    \put(0.25369717,0.32919105){\color[rgb]{0,0,0}\makebox(0,0)[lb]{\smash{$x$}}}%
    \put(0.25271788,0.19046458){\color[rgb]{0,0,0}\makebox(0,0)[lb]{\smash{$x$}}}%
    \put(0.20277635,0.19079099){\color[rgb]{0,0,0}\makebox(0,0)[lb]{\smash{$x$}}}%
    \put(0.15218198,0.19209666){\color[rgb]{0,0,0}\makebox(0,0)[lb]{\smash{$x$}}}%
    \put(0.14566457,0.2601634){\color[rgb]{0,0,0}\makebox(0,0)[lb]{\smash{$h_1$}}}%
    \put(0.21584385,0.25787848){\color[rgb]{0,0,0}\makebox(0,0)[lb]{\smash{$h_2$}}}%
    \put(0.0879275,0.36624836){\color[rgb]{0,0,0}\makebox(0,0)[lb]{\smash{. . .}}}%
    \put(0.25436086,0.36526914){\color[rgb]{0,0,0}\makebox(0,0)[lb]{\smash{. . .}}}%
    \put(0.25403442,0.1371048){\color[rgb]{0,0,0}\makebox(0,0)[lb]{\smash{. . .}}}%
    \put(0.08890675,0.14004256){\color[rgb]{0,0,0}\makebox(0,0)[lb]{\smash{. . .}}}%
    \put(0.25318824,0.49238002){\color[rgb]{0,0,0}\makebox(0,0)[lb]{\smash{. . .}}}%
    \put(0.08871336,0.49368567){\color[rgb]{0,0,0}\makebox(0,0)[lb]{\smash{. . .}}}%
    \put(0.08871336,0.04127405){\color[rgb]{0,0,0}\makebox(0,0)[lb]{\smash{. . .}}}%
    \put(0.25449387,0.0393156){\color[rgb]{0,0,0}\makebox(0,0)[lb]{\smash{. . .}}}%
    \put(0.59194638,0.34752468){\color[rgb]{0,0,0}\makebox(0,0)[lb]{\smash{. . . }}}%
    \put(0.88825663,0.34752468){\color[rgb]{0,0,0}\makebox(0,0)[lb]{\smash{. . .}}}%
    \put(0.51507043,0.44302417){\color[rgb]{0,0,0}\makebox(0,0)[lb]{\smash{. . .}}}%
    \put(0.96091471,0.43943692){\color[rgb]{0,0,0}\makebox(0,0)[lb]{\smash{. . .}}}%
    \put(0.87922144,0.49990775){\color[rgb]{0,0,0}\makebox(0,0)[lb]{\smash{. . .}}}%
    \put(0.73870036,0.50042021){\color[rgb]{0,0,0}\makebox(0,0)[lb]{\smash{. . .}}}%
    \put(0.60699695,0.49119584){\color[rgb]{0,0,0}\makebox(0,0)[lb]{\smash{. . .}}}%
    \put(0.88813875,0.18192882){\color[rgb]{0,0,0}\makebox(0,0)[lb]{\smash{. . .}}}%
    \put(0.96079682,0.09001659){\color[rgb]{0,0,0}\makebox(0,0)[lb]{\smash{. . .}}}%
    \put(0.87807234,0.03428172){\color[rgb]{0,0,0}\makebox(0,0)[lb]{\smash{. . .}}}%
    \put(0.73966687,0.02969577){\color[rgb]{0,0,0}\makebox(0,0)[lb]{\smash{. . .}}}%
    \put(0.60687901,0.03825768){\color[rgb]{0,0,0}\makebox(0,0)[lb]{\smash{. . .}}}%
    \put(0.59222655,0.18079345){\color[rgb]{0,0,0}\makebox(0,0)[lb]{\smash{. . . }}}%
    \put(0.51478812,0.08293444){\color[rgb]{0,0,0}\makebox(0,0)[lb]{\smash{. . .}}}%
    \put(0.8746591,0.26002312){\color[rgb]{0,0,0}\makebox(0,0)[lb]{\smash{$\Gamma_H$}}}%
    \put(0.7766041,0.40793664){\color[rgb]{0,0,0}\makebox(0,0)[lb]{\smash{$x\Gamma_H$}}}%
    \put(0.77007283,0.11322256){\color[rgb]{0,0,0}\makebox(0,0)[lb]{\smash{$x^{-1}\Gamma_H$}}}%
    \put(0.68283493,0.34458917){\color[rgb]{0,0,0}\makebox(0,0)[lb]{\smash{$x$}}}%
    \put(0.86305141,0.18713953){\color[rgb]{0,0,0}\makebox(0,0)[lb]{\smash{$x$}}}%
    \put(0.77138776,0.18185126){\color[rgb]{0,0,0}\makebox(0,0)[lb]{\smash{$x$}}}%
    \put(0.68060537,0.18185126){\color[rgb]{0,0,0}\makebox(0,0)[lb]{\smash{$x$}}}%
    \put(0.86187617,0.33844342){\color[rgb]{0,0,0}\makebox(0,0)[lb]{\smash{$x$}}}%
    \put(0.77226908,0.34520068){\color[rgb]{0,0,0}\makebox(0,0)[lb]{\smash{$x$}}}%
    \put(0.76252196,0.25499366){\color[rgb]{0,0,0}\makebox(0,0)[lb]{\smash{$1$}}}%
    \put(0.10997274,0.44342176){\color[rgb]{0,0,0}\makebox(0,0)[lb]{\smash{$xh_1$}}}%
    \put(0.24175036,0.44342176){\color[rgb]{0,0,0}\makebox(0,0)[lb]{\smash{$xh_2$}}}%
  \end{picture}%
\endgroup

%% file: fig3.pdf_tex

\begingroup
  \makeatletter
  \providecommand\color[2][]{%
    \errmessage{(Inkscape) Color is used for the text in Inkscape, but the package 'color.sty' is not loaded}
    \renewcommand\color[2][]{}%
  }
  \providecommand\transparent[1]{%
    \errmessage{(Inkscape) Transparency is used (non-zero) for the text in Inkscape, but the package 'transparent.sty' is not loaded}
    \renewcommand\transparent[1]{}%
  }
  \providecommand\rotatebox[2]{#2}
  \ifx\svgwidth\undefined
    \setlength{\unitlength}{313.59211304pt}
  \else
    \setlength{\unitlength}{\svgwidth}
  \fi
  \global\let\svgwidth\undefined
  \makeatother
  \begin{picture}(1,0.37684302)%
    \put(0,0){\includegraphics[width=\unitlength]{fig3.pdf}}%
    \put(0.46009676,0.16353481){\color[rgb]{0,0,0}\makebox(0,0)[lb]{\smash{$e$}}}%
    \put(0.52375461,0.06322187){\color[rgb]{0,0,0}\makebox(0,0)[lb]{\smash{$a$}}}%
    \put(0.5078744,0.30686757){\color[rgb]{0,0,0}\makebox(0,0)[lb]{\smash{$b$}}}%
    \put(0.77797367,0.3453927){\color[rgb]{0,0,0}\makebox(0,0)[lb]{\smash{$q_2$}}}%
    \put(0.786316,0.1463334){\color[rgb]{0,0,0}\makebox(0,0)[lb]{\smash{$p_2$}}}%
    \put(0.27654534,0.2780217){\color[rgb]{0,0,0}\makebox(0,0)[lb]{\smash{$q_{in}(xH_\lambda)$}}}%
    \put(0.29982158,0.03274486){\color[rgb]{0,0,0}\makebox(0,0)[lb]{\smash{$p_{in}(xH_\lambda)$}}}%
    \put(0.639768,0.10134489){\color[rgb]{0,0,0}\makebox(0,0)[lb]{\smash{$p_{out}(xH_\lambda)$}}}%
    \put(0.60483492,0.3424506){\color[rgb]{0,0,0}\makebox(0,0)[lb]{\smash{$q_{out}(xH_\lambda)$}}}%
    \put(0.21517015,0.04354501){\color[rgb]{0,0,0}\makebox(0,0)[lb]{\smash{$p_1$}}}%
    \put(0.22083099,0.22856411){\color[rgb]{0,0,0}\makebox(0,0)[lb]{\smash{$q_1$}}}%
    \put(0.53128158,0.19208823){\color[rgb]{0,0,0}\makebox(0,0)[lb]{\smash{$xH_\lambda$}}}%
  \end{picture}%
\endgroup

%% file: fig1.pdf_tex

\begingroup
  \makeatletter
  \providecommand\color[2][]{%
    \errmessage{(Inkscape) Color is used for the text in Inkscape, but the package 'color.sty' is not loaded}
    \renewcommand\color[2][]{}%
  }
  \providecommand\transparent[1]{%
    \errmessage{(Inkscape) Transparency is used (non-zero) for the text in Inkscape, but the package 'transparent.sty' is not loaded}
    \renewcommand\transparent[1]{}%
  }
  \providecommand\rotatebox[2]{#2}
  \ifx\svgwidth\undefined
    \setlength{\unitlength}{401.94364014pt}
  \else
    \setlength{\unitlength}{\svgwidth}
  \fi
  \global\let\svgwidth\undefined
  \makeatother
  \begin{picture}(1,0.29933655)%
    \put(0,0){\includegraphics[width=\unitlength]{fig1.pdf}}%
    \put(0.06946207,0.1306944){\color[rgb]{0,0,0}\makebox(0,0)[lb]{\smash{$p_1$}}}%
    \put(0.20667397,0.19558863){\color[rgb]{0,0,0}\makebox(0,0)[lb]{\smash{$a_1$}}}%
    \put(0.30421606,0.21655065){\color[rgb]{0,0,0}\makebox(0,0)[lb]{\smash{$p_2$}}}%
    \put(0.40776955,0.20251981){\color[rgb]{0,0,0}\makebox(0,0)[lb]{\smash{$a_2$}}}%
    \put(0.57003444,0.17964036){\color[rgb]{0,0,0}\makebox(0,0)[lb]{\smash{.     .     .}}}%
    \put(0.77025199,0.14077381){\color[rgb]{0,0,0}\makebox(0,0)[lb]{\smash{$a_n$}}}%
    \put(0.86077222,0.17653106){\color[rgb]{0,0,0}\makebox(0,0)[lb]{\smash{$p_{n+1}$}}}%
    \put(0.96444231,0.17381038){\color[rgb]{0,0,0}\makebox(0,0)[lb]{\smash{$g$}}}%
    \put(-0.00100294,0.07081398){\color[rgb]{0,0,0}\makebox(0,0)[lb]{\smash{$f$}}}%
    \put(0.19973865,0.06808799){\color[rgb]{0,0,0}\makebox(0,0)[lb]{\smash{$x_1H_\lambda$}}}%
    \put(0.39222977,0.06749201){\color[rgb]{0,0,0}\makebox(0,0)[lb]{\smash{$x_2H_\lambda $}}}%
    \put(0.7563537,0.06451229){\color[rgb]{0,0,0}\makebox(0,0)[lb]{\smash{$x_nH_\lambda$}}}%
  \end{picture}%
\endgroup

%% file: fig2.pdf_tex

\begingroup
  \makeatletter
  \providecommand\color[2][]{%
    \errmessage{(Inkscape) Color is used for the text in Inkscape, but the package 'color.sty' is not loaded}
    \renewcommand\color[2][]{}%
  }
  \providecommand\transparent[1]{%
    \errmessage{(Inkscape) Transparency is used (non-zero) for the text in Inkscape, but the package 'transparent.sty' is not loaded}
    \renewcommand\transparent[1]{}%
  }
  \providecommand\rotatebox[2]{#2}
  \ifx\svgwidth\undefined
    \setlength{\unitlength}{348.46491089pt}
  \else
    \setlength{\unitlength}{\svgwidth}
  \fi
  \global\let\svgwidth\undefined
  \makeatother
  \begin{picture}(1,0.45618518)%
    \put(0,0){\includegraphics[width=\unitlength]{fig2.pdf}}%
    \put(0.59637634,0.22241589){\color[rgb]{0,0,0}\makebox(0,0)[lb]{\smash{$e$}}}%
    \put(0.3839208,0.26736299){\color[rgb]{0,0,0}\makebox(0,0)[lb]{\smash{$r_1$}}}%
    \put(0.67419204,0.40131039){\color[rgb]{0,0,0}\makebox(0,0)[lb]{\smash{$r_2$}}}%
    \put(0.58248919,0.34585724){\color[rgb]{0,0,0}\makebox(0,0)[lb]{\smash{$b$}}}%
    \put(0.57014501,0.14796533){\color[rgb]{0,0,0}\makebox(0,0)[lb]{\smash{$a_2$}}}%
    \put(0.42162979,0.10128907){\color[rgb]{0,0,0}\makebox(0,0)[lb]{\smash{$p_2$}}}%
    \put(0.28237246,0.14372204){\color[rgb]{0,0,0}\makebox(0,0)[lb]{\smash{$a_1$}}}%
    \put(0.12845653,0.08547319){\color[rgb]{0,0,0}\makebox(0,0)[lb]{\smash{$p_1$}}}%
    \put(0.75260684,0.09704578){\color[rgb]{0,0,0}\makebox(0,0)[lb]{\smash{$p_3$}}}%
    \put(0.27118559,0.07428632){\color[rgb]{0,0,0}\makebox(0,0)[lb]{\smash{$x_jH_\lambda$}}}%
    \put(0.55818666,0.07428625){\color[rgb]{0,0,0}\makebox(0,0)[lb]{\smash{$x_iH_\lambda$}}}%
    \put(0.95898529,0.10244635){\color[rgb]{0,0,0}\makebox(0,0)[lb]{\smash{$g$}}}%
    \put(0.7788381,0.43689521){\color[rgb]{0,0,0}\makebox(0,0)[lb]{\smash{$h$}}}%
    \put(0.85135986,0.25790528){\color[rgb]{0,0,0}\makebox(0,0)[lb]{\smash{$q$}}}%
    \put(-0.00115686,0.10321783){\color[rgb]{0,0,0}\makebox(0,0)[lb]{\smash{$f$}}}%
  \end{picture}%
\endgroup

%% file: fig4.pdf_tex

\begingroup
  \makeatletter
  \providecommand\color[2][]{%
    \errmessage{(Inkscape) Color is used for the text in Inkscape, but the package 'color.sty' is not loaded}
    \renewcommand\color[2][]{}%
  }
  \providecommand\transparent[1]{%
    \errmessage{(Inkscape) Transparency is used (non-zero) for the text in Inkscape, but the package 'transparent.sty' is not loaded}
    \renewcommand\transparent[1]{}%
  }
  \providecommand\rotatebox[2]{#2}
  \ifx\svgwidth\undefined
    \setlength{\unitlength}{191.64634399pt}
  \else
    \setlength{\unitlength}{\svgwidth}
  \fi
  \global\let\svgwidth\undefined
  \makeatother
  \begin{picture}(1,0.90064207)%
    \put(0,0){\includegraphics[width=\unitlength]{fig4.pdf}}%
    \put(0.46344848,0.37278558){\color[rgb]{0,0,0}\makebox(0,0)[lb]{\smash{$xH_\lambda$}}}%
    \put(-0.00210348,0.01313353){\color[rgb]{0,0,0}\makebox(0,0)[lb]{\smash{$f$}}}%
    \put(0.90883667,0.01081093){\color[rgb]{0,0,0}\makebox(0,0)[lb]{\smash{$g$}}}%
    \put(0.52420932,0.8655677){\color[rgb]{0,0,0}\makebox(0,0)[lb]{\smash{$h$}}}%
    \put(0.14925847,0.04925887){\color[rgb]{0,0,0}\makebox(0,0)[lb]{\smash{$p_1$}}}%
    \put(0.06746818,0.18937604){\color[rgb]{0,0,0}\makebox(0,0)[lb]{\smash{$r_1$}}}%
    \put(0.40482004,0.7230702){\color[rgb]{0,0,0}\makebox(0,0)[lb]{\smash{$r_2$}}}%
    \put(0.76137878,0.04893479){\color[rgb]{0,0,0}\makebox(0,0)[lb]{\smash{$p_2$}}}%
    \put(0.38031685,0.26621142){\color[rgb]{0,0,0}\makebox(0,0)[lb]{\smash{$e_1$}}}%
    \put(0.50347703,0.18231971){\color[rgb]{0,0,0}\makebox(0,0)[lb]{\smash{$a$}}}%
    \put(0.37942441,0.46937698){\color[rgb]{0,0,0}\makebox(0,0)[lb]{\smash{$b$}}}%
    \put(0.62905522,0.45331263){\color[rgb]{0,0,0}\makebox(0,0)[lb]{\smash{$c$}}}%
    \put(0.50914836,0.55237625){\color[rgb]{0,0,0}\makebox(0,0)[lb]{\smash{$e_2$}}}%
    \put(0.6302071,0.26710392){\color[rgb]{0,0,0}\makebox(0,0)[lb]{\smash{$e_3$}}}%
    \put(0.87533189,0.17391144){\color[rgb]{0,0,0}\makebox(0,0)[lb]{\smash{$q_2$}}}%
    \put(0.61205455,0.71831536){\color[rgb]{0,0,0}\makebox(0,0)[lb]{\smash{$q_1$}}}%
  \end{picture}%
\endgroup

%% file: fig5.pdf_tex

\begingroup
  \makeatletter
  \providecommand\color[2][]{%
    \errmessage{(Inkscape) Color is used for the text in Inkscape, but the package 'color.sty' is not loaded}
    \renewcommand\color[2][]{}%
  }
  \providecommand\transparent[1]{%
    \errmessage{(Inkscape) Transparency is used (non-zero) for the text in Inkscape, but the package 'transparent.sty' is not loaded}
    \renewcommand\transparent[1]{}%
  }
  \providecommand\rotatebox[2]{#2}
  \ifx\svgwidth\undefined
    \setlength{\unitlength}{188.46745605pt}
  \else
    \setlength{\unitlength}{\svgwidth}
  \fi
  \global\let\svgwidth\undefined
  \makeatother
  \begin{picture}(1,0.91583311)%
    \put(0,0){\includegraphics[width=\unitlength]{fig5.pdf}}%
    \put(0.354542,0.29541013){\color[rgb]{0,0,0}\makebox(0,0)[lb]{\smash{$xH_\lambda $}}}%
    \put(-0.00213896,0.01335505){\color[rgb]{0,0,0}\makebox(0,0)[lb]{\smash{$f$}}}%
    \put(0.92416628,0.01099328){\color[rgb]{0,0,0}\makebox(0,0)[lb]{\smash{$g$}}}%
    \put(0.53305127,0.88016714){\color[rgb]{0,0,0}\makebox(0,0)[lb]{\smash{$h$}}}%
    \put(0.15363074,0.05054286){\color[rgb]{0,0,0}\makebox(0,0)[lb]{\smash{$p_1$}}}%
    \put(0.42679381,0.18485556){\color[rgb]{0,0,0}\makebox(0,0)[lb]{\smash{$a$}}}%
    \put(0.70721707,0.06869347){\color[rgb]{0,0,0}\makebox(0,0)[lb]{\smash{$p_2$}}}%
    \put(0.74805544,0.48524443){\color[rgb]{0,0,0}\makebox(0,0)[lb]{\smash{$q$}}}%
    \put(0.07046079,0.19302332){\color[rgb]{0,0,0}\makebox(0,0)[lb]{\smash{$r_1$}}}%
    \put(0.32515172,0.38541733){\color[rgb]{0,0,0}\makebox(0,0)[lb]{\smash{$b$}}}%
    \put(0.41350288,0.73571948){\color[rgb]{0,0,0}\makebox(0,0)[lb]{\smash{$r_2$}}}%
    \put(0.27705328,0.21208123){\color[rgb]{0,0,0}\makebox(0,0)[lb]{\smash{$e_1$}}}%
    \put(0.52752834,0.37543472){\color[rgb]{0,0,0}\makebox(0,0)[lb]{\smash{$e$}}}%
  \end{picture}%
\endgroup